\documentclass[12pt]{amsart}

\usepackage{amsmath, amssymb, amscd, newlfont, bbm}
\usepackage[foot]{amsaddr}
\usepackage[final]{hyperref}
\hypersetup{
	colorlinks=true,        
	linkcolor={black},  
	citecolor={black},
	urlcolor={black}     
}
\usepackage{enumitem}
\usepackage{tikz-cd}
\usepackage{comment}

\newcommand{\spacedcdot}{{\,\cdot\,}}

\newcommand{\bbA}{{\mathbb{A}}}

\newcommand{\bbQ}{{\mathbb{Q}}}
\newcommand{\bbC}{{\mathbb{C}}}

\newcommand{\bbR}{{\mathbb{R}}}
\newcommand{\bbZ}{{\mathbb{Z}}}

\newcommand{\End}{{\mathrm{End}}}

\DeclareMathOperator{\trace}{{\mathbf{tr}}}

\DeclareMathOperator{\Cl}{{\mathrm{Cl}}}
\DeclareMathOperator{\supp}{{\mathrm{supp}}}
\DeclareMathOperator{\Ad}{{\mathrm{Ad}}}
\DeclareMathOperator{\vol}{{\mathrm{vol}}}
\DeclareMathOperator{\aconv}{{\mathrm{AConv}}}

\newcommand{\calA}{{\mathcal{A}}}
\newcommand{\calC}{{\mathcal{C}}}

\newcommand{\calG}{{\mathcal{G}}}
\newcommand{\calH}{{\mathcal{H}}}

\newcommand{\calO}{{\mathcal{O}}}
\newcommand{\calS}{{\mathcal{S}}}

\newcommand{\calU}{{\mathcal{U}}}

\newcommand{\frakg}{{\mathfrak{g}}}

\usepackage[mathscr]{eucal}

\def\SL#1{{\mathrm{SL}}_{#1}}

\providecommand{\abs}[1]{\left\lvert#1\right\rvert}
\providecommand{\norm}[1]{\left\lVert#1\right\rVert}
\providecommand{\scal}[2]{\left<#1,#2\right>}

\numberwithin{equation}{section}

\newtheorem{Prop}[equation]{Proposition}
\newtheorem{Lem}[equation]{Lemma}
\newtheorem{Thm}[equation]{Theorem}
\newtheorem{Cor}[equation]{Corollary}

\theoremstyle{definition}
\newtheorem{Def}[equation]{Definition}

\theoremstyle{remark}
\newtheorem{Rem}[equation]{Remark}

\title{ On the Schwartz Space $ \calS(G({\MakeLowercase k})\backslash G(\bbA)) $ }

\author{Goran Mui\'c and Sonja \v Zunar}

\address{ Department of Mathematics,
	Faculty of Science,
	University of Zagreb,
	Bijeni\v cka 30,
	10000 Zagreb,
	Croatia}
\email{gmuic@math.hr}
\email{szunar@math.hr}

\thanks{The authors acknowledge Croatian Science Foundation grant No.~3628.}

\begin{document}
\maketitle

\begin{abstract}
	For a connected reductive group $ G $ defined over a number field $ k $, we construct the Schwartz space $ \calS(G(k)\backslash G(\bbA)) $. This space is an adelic version of Casselman's Schwartz space $ \calS(\Gamma\backslash G_\infty) $, where $ \Gamma $ is a discrete subgroup of $ G_\infty:=\prod_{v\in V_\infty}G(k_v) $. We study the space of tempered distributions $ \calS(G(k)\backslash G(\bbA))' $ and investigate applications to automorphic forms on $ G(\bbA) $. In particular, we study the representation $ \left(r',\calS(G(k)\backslash G(\bbA))'\right) $ contragredient to the right regular representation $ (r,\calS(G(k)\backslash G(\bbA))) $ of $ G(\bbA) $ and describe the closed irreducible admissible subrepresentations of $ \calS(G(k)\backslash G(\bbA))' $ assuming that $ G $ is semisimple.
\end{abstract}

\section{Introduction}

Let $ G_\infty $ be the group of $ \bbR $-rational points of a reductive group defined over $ \bbQ $, and let $ \Gamma $ be a discrete subgroup of $ G_\infty $. In \cite{casselman}, Casselman introduced the Schwartz space $ \calS(\Gamma\backslash G_\infty) $ and used it to study the cohomology of arithmetic subgroups. The right regular representation $ r_\Gamma $ of $ G_\infty $ on $ \calS(\Gamma\backslash G_\infty) $ is a smooth Fr\'echet representation of moderate growth, and the contragredient representation $ r_\Gamma' $ on the strong dual $ \calS(\Gamma\backslash G_\infty)' $ has the following remarkable property: the G\aa rding subspace of $ \calS(\Gamma\backslash G_\infty)' $ is the space $ \calA_{umg}(\Gamma\backslash G_\infty) $ of functions in $ C^\infty(\Gamma\backslash G_\infty) $ of uniformly moderate growth \cite[Theorem 1.16]{casselman}. Thus, the standardly defined $ K_\infty $-finite automorphic forms $ \varphi\in\calA(\Gamma\backslash G_\infty) $ (see \eqref{eq:098}; cf.\ \cite[\S1]{BJ}) are simply the $ Z(\frakg_\infty) $-finite and $ K_\infty $-finite vectors in the G\aa rding subspace of $ \calS(\Gamma\backslash G_\infty)' $. Here $ K_\infty $ is a maximal compact subgroup of $ G_\infty $, and $ Z(\frakg_\infty) $ is the center of the universal enveloping algebra of the complexified Lie algebra of $ G_\infty $. 

\vskip .2in
 
In this paper, we extend Casselman's construction of the Schwartz space $ \calS(\Gamma\backslash G_\infty) $ to the adelic setting. More precisely, let $ G $ be a (Zariski) connected reductive group defined over a number field $ k $. Let $ V_\infty $ (resp., $ V_f $) be the set of archimedean (resp., non-trivial non-archimedean) places of $ k $. Let $ k_v $ be the completion of $ k $ at a place $ v\in V_\infty\cup V_f $, and let $ \bbA $ (resp., $ \bbA_f $) denote the ring of adeles (resp., finite adeles) of $ k $. We embed the group $ G(k) $ diagonally into the groups $ G_\infty:=\prod_{v\in V_\infty}G(k_v) $, $ G(\bbA_f)=\prod_{v\in V_f}'G(k_v) $ and $ G(\bbA)=G_\infty\times G(\bbA_f) $. Here $ \prod'_{v\in V_f} $ denotes the restricted product, over all $ v\in V_f $, with respect to certain open compact subgroups $ K_v $ of $ G(k_v) $. For an open compact subgroup $ L $ of $ G(\bbA_f) $ and $ c\in G(\bbA_f) $, let $ \Gamma_{c,L} $ be the congruence subgroup of $ G(k) $ that, embedded diagonally into $ G(\bbA_f) $, equals $ G(k)\cap cLc^{-1} $. In Section \ref{sec:104}, we construct a Fr\'echet space $ \calS^L\subseteq C^\infty(G(k)\backslash G(\bbA)) $ with the following properties: 
\begin{enumerate}[label=\textup{(\arabic*)},leftmargin=*,align=left]
	\item Every $ f\in\calS^L $ is $ L $-invariant on the right.
	\item\label{enum:105:2} For every (finite) set $ C\subseteq G(\bbA_f) $ such that $ G(\bbA_f)=\bigsqcup_{c\in C}G(k)cL $, $ \calS^L $ is isomorphic to the direct sum of Fr\'echet spaces
	\[ \bigoplus_{c\in C}\calS(\Gamma_{c,L}\backslash G_\infty) \]
	via the map $ f\mapsto \left(f(\spacedcdot, c)\right)_{c\in C} $.
	\item\label{enum:105:3} If $ L'\subseteq L $, then $ \calS^{L} $ is a closed subspace of $ \calS^{L'} $.	
\end{enumerate}
The decomposition \ref{enum:105:2} is analogous to the well-known decomposition \cite[\S4.3(2)]{BJ} of spaces of adelic automorphic forms, and it allows for a simple adelic reformulation of the above-mentioned Casselman's results (see Lemma \hyperref[lem:023:1]{\ref*{lem:023}\ref*{lem:023:1}} and Corollary \ref{cor:073}). 

\vskip .2in

The property \ref{enum:105:3}  enables us to define the Schwartz space $ \calS=\calS(G(k)\backslash G(\bbA)) $ as the strict inductive limit of Fr\'echet spaces $ \calS^L $. In other words, we put 
\[ \calS:=\bigcup_{L}\calS^L \]
and equip $ \calS $ with the finest locally convex topology with respect to which the inclusion maps $ \calS^L\hookrightarrow\calS $ are continuous. The space $ \calS $ is not a Fr\'echet space; however, it is an LF-space, i.e., a strict inductive limit of an increasing sequence of Fr\'echet subspaces (see Definition \ref{def:045}). In particular, $ \calS $ is a complete locally convex (Hausdorff) topological vector space, and we have the following description of bounded sets in $ \calS $ (see Definition \ref{def:017}): a subset $ B $ of $ \calS $ is bounded in $ \calS $ if and only if $ B $ is a bounded subset of $ \calS^L $ for some $ L $ (Lemma \hyperref[lem:041:3]{\ref*{lem:041}\ref*{lem:041:3}}). This description clarifies the structure of the strong dual $ \calS' $---the space of continuous linear functionals $ \calS\to\bbC $ equipped with the topology of uniform convergence on bounded sets in $ \calS $. By analogy with Casselman's terminology in \cite{casselman}, we call $ \calS' $ the space of tempered distributions on $ G(k)\backslash G(\bbA) $.

\vskip .2in

The space $ \calS' $ is our main object of interest. It can be identified with the projective limit of the spaces $ \left(\calS^L\right)' $, as described in Lemma \ref{lem:107}. The right regular representation $ (r,\calS) $ of $ G(\bbA) $ and its contragredient representation $ \left(r',\calS'\right) $ are continuous representations of $ G(\bbA) $ and smooth representations of $ G_\infty $ (Propositions \hyperref[prop:016:1]{\ref*{prop:016}} and \hyperref[prop:044:1]{\ref*{prop:044}}). 

\vskip .2in

In Section \ref{sec:152}, we recall some classical results on the action of Hecke algebra $ \calH(G(\bbA_f)) $ on $ V $, where $ (\pi,V) $ is a continuous representation of $ G(\bbA) $ on a complete complex locally convex topological vector space $ V $. Using these results, we study the subspaces of $ L $-invariants in $ \left(r,\calS\right) $ and $ \left(r',\calS'\right) $, where $ L $ is an open compact subgroup of $ G(\bbA_f) $ (see Lemmas \hyperref[lem:138:1]{\ref*{lem:138}} and \ref{lem:144}). 

\vskip .2in

In Section \ref{sec:117}, we use results of Sections \ref{sec:104} and \ref{sec:152} to prove that the G\aa rding subspace of $ \calS' $ coincides with the space $ \calA_{umg}(G(k)\backslash G(\bbA)) $ of functions in $ C^\infty(G(k)\backslash G(\bbA)) $ of uniformly moderate growth (see \eqref{eq:119} and \eqref{eq:054}). This is the adelic version (Theorem \hyperref[thm:145:2]{\ref*{thm:145}\ref*{thm:145:2}}) of the above-mentioned Casselman's result \cite[Theorem 1.16]{casselman}. Here the G\aa rding subspace of a continuous representation $ \pi $ of $ G(\bbA) $ on a complete complex locally convex topological vector space $ V $ is defined by analogy with the classical definition (see \eqref{eq:086}) of the G\aa rding subspace for representations of Lie groups: we put
\[ V_{G(\bbA)\text{-G\aa rding}}:=\mathrm{span}_\bbC\left\{\pi(\alpha)v:\alpha\in C_c^\infty(G_\infty)\otimes C_c^\infty(G(\bbA_f)),\ v\in V\right\}, \]
where the convolution algebra $ C_c(G(\bbA)) $ and its subalgebra $ C_c^\infty(G_\infty)\otimes C_c^\infty(G(\bbA_f)) $ act on $ V $ in the following standard way:
\[ \pi(\alpha)v:=\int_{G(\bbA)}\alpha(x)\,\pi(x)v\,dx,\qquad\alpha\in C_c(G(\bbA)),\ v\in V, \]
where $ dx $ is a Haar measure on $ G(\bbA) $. As in the classical situation, the G\aa rding subspace $ V_{G(\bbA)\text{-G\aa rding}} $ is dense in $ V $ (see the discussion after \eqref{eq:087}), and it is contained in the subspace $ V^\infty $ of $ G_\infty $- and $ G(\bbA_f) $-smooth vectors in $ V $ (Corollary \hyperref[cor:141:1]{\ref*{cor:141}\ref*{cor:141:1}}); in particular, $ V^\infty $ is dense in $ V $.

\vskip .2in

In Section \ref{sec:131}, we assume that $ G $ is semisimple and, under this assumption, describe closed irreducible admissible subrepresentations of $ \left(r',\calS'\right) $: in Theorem \ref{thm:126} we prove that they are the closures in $ \calS' $ of irreducible (admissible) $ (\frakg_\infty,K_\infty)\times G(\bbA_f) $-submodules of $ \calA(G(k)\backslash G(\bbA)) $. Here the notion of admissibility is defined classically (Definitions \ref{def:152} and \ref{def:129}; cf.~\cite{flath} and \cite[\S4]{BJ}). In relation to this, we note that the closed irreducible admissible subrepresentations of Casselman's Schwartz space $ \calS(\Gamma\backslash G_\infty)' $ were studied in \cite{muicRadHAZU} and \cite{muicMathZ}. 

\vskip .2in

In Sections \ref{sec:107} and \ref{sec:108}, we turn our attention to Poincar\'e series on $ G(\bbA) $.
In Section \ref{sec:107}, we construct an LF-space $ \calS(G(\bbA)) $ such that the linear operator $ P_{G(k)}:\calS(G(\bbA))\to\calS $,
\[ P_{G(k)}f:=\sum_{\delta\in G(k)}f(\delta\spacedcdot), \]
is well-defined, continuous and surjective (Theorem \ref{thm:043}). The proof of Theorem \ref{thm:043} is based on an analogous result \cite[Proposition 1.11 and Theorem 2.2]{casselman} for the operator $ P_\Gamma:\calS(G_\infty)\to\calS(\Gamma\backslash G_\infty) $, where $ \calS(G_\infty):=\calS(\left\{1_{G_\infty}\right\}\backslash G_\infty) $ and $ \Gamma $ is an arithmetic subgroup. 

\vskip .2in

In Section \ref{sec:108}, we use results of previous sections to study the Poincar\'e series $ P_{G(k)}\varphi:=\sum_{\delta\in G(k)}\varphi(\delta\spacedcdot) $ of functions $ \varphi\in L^1(G(\bbA)) $. The main results of this section---Propositions \ref{prop:090} and \ref{prop:091}---may be regarded as the adelic version of \cite[Proposition 6.4]{muicRadHAZU}. Let us describe them in more detail.

\vskip .2in

It is well-known that for every $ \varphi\in L^1(G(\bbA)) $, the series $ P_{G(k)}\varphi $ converges absolutely almost everywhere on $ G(\bbA) $ and defines an element of $ L^1(G(k)\backslash G(\bbA)) $ (e.g., see \cite[\S4]{muicMathAnn}). The functions $ \varphi(\delta\spacedcdot) $, $ \delta\in G(k) $, may be regarded as elements of the strong dual $ \calS(G(\bbA))' $ in a standard way (see \eqref{eq:109}), and in Proposition \ref{prop:091} we study the convergence of the series $ P_{G(k)}\varphi $ in $ \calS(G(\bbA))' $. 

\vskip .2in

Next, let $ r_{L^1(G(\bbA))} $ be the right regular representation of $ G(\bbA) $ on $ L^1(G(\bbA)) $. In Proposition \ref{prop:090}, we use $ r_{L^1(G(\bbA))} $ to describe a construction of adelic automorphic forms using Poincar\'e series. More precisely, we prove that for every $ \varphi\in L^1(G(\bbA))_{G(\bbA)\text{-G\aa rding}} $, the function $ P_{G(k)}\varphi\in L^1(G(k)\backslash G(\bbA)) $ coincides almost everywhere with an element of $ \calA_{umg}(G(k)\backslash G(\bbA)) $. This means that if $ \varphi $ is additionally a $ K_\infty $-finite and $ Z(\frakg_\infty) $-finite vector in $ L^1(G(\bbA)) $, then $ P_{G(k)}\varphi $ is a $ K_\infty $-finite automorphic form on $ G(\bbA) $. Let us mention that the question when the constructed function $ P_{G(k)}\varphi\in\calA_{umg}(G(k)\backslash G(\bbA)) $ vanishes identically is non-trivial. It may be studied using the integral non-vanishing criterion \cite[Theorem 4-1]{muicMathAnn} for Poincar\'e series on unimodular locally compact Hausdorff groups. This criterion was strengthened in \cite[Lemma 2-1]{muicIJNT} and \cite[Theorem 1]{zunarRama}, and it found applications to various classes of cuspidal automorphic forms on $ \SL2(\bbR) $ (resp., on the metaplectic cover of $ \SL2(\bbR) $) in \cite{muicJNT}, \cite{muicIJNT} and \cite{muicLFunk} (resp., in \cite{zunarGlas}, \cite{zunarManu} and \cite{zunarRama}).

\vskip .2in

We end the paper by an appendix in which we collect some well-known facts from functional analysis used in the paper. 

\vskip .2in

In future work, we will study applications to compactly supported Poincar\'e series related to results of \cite{moyMuicTAMS} and \cite{muicCompositio}. We expect our results to find various interesting applications, e.g.\ in cohomology theory, especially to the Eisenstein cohomology of adelic reductive groups (see \cite{grbac}, \cite{grobner}, \cite{ghl}, \cite{liSchwermer} and \cite{schwermer}).

\vskip .2in

We would like to thank H.~Grobner for enlightening discussions on analytic theory of automorphic forms \cite{grobnernotes}. In particular, the second author would like to thank H.~Grobner and the University of Vienna for their hospitality during her visit in May 2019.

\section{Preliminaries}\label{sec:023}                                                                                                                                   

Let $ G $ be a (Zariski) connected reductive group defined over a number field $ k $.
Let $ V_\infty $ (resp., $ V_f $) be the set of archimedean (resp., non-trivial non-archimedean) places of $ k $. Let
$ k_v $ be the completion of $ k $ at a place $ v\in V_\infty\cup V_f $, and let $ \calO_v $ be the ring of integers of $ k_v $. 

\vskip .2in

The group $ G_\infty:=\prod_{v\in V_\infty}G(k_v) $ is a reductive Lie group. Let $ \frakg_\infty $ be the Lie algebra of $G_\infty$, let $ U(\frakg_\infty) $ be the universal enveloping algebra of the complexification of $ \frakg_\infty $, and let $ Z(\frakg_\infty) $ be the center of $ U(\frakg_\infty) $. We fix a maximal compact subgroup $ K_\infty $ of $ G_\infty $.

\vskip .2in 

By a norm (see \cite[\S1.2]{BJ}) on $ G_\infty $ we mean a function $ \norm\spacedcdot:G_\infty\to\bbR_{\geq1} $ of the form

\[
\norm x=\trace{\sigma^*(x)\sigma(x)}+\trace{\sigma^*(x^{-1})\sigma(x^{-1})},
\]
where $ \sigma:G_\infty\to GL(E) $ is a continuous representation of $ G_\infty $ with finite kernel on a finite-dimensional complex Hilbert space $ E $
such that $ \sigma\big|_{K_\infty} $ is unitary. Here $ ^* $ denotes the adjoint on $ \End(E) $ with respect to the Hilbert space structure on $ E $. The following properties of
$ \norm\spacedcdot $ are obvious:

\begin{enumerate}[label=\textup{(N\arabic*)},leftmargin=*,align=left]
	\item\label{enum:001:1} $ \norm{x^{-1}}=\norm x $ for all $ x\in G_\infty $.
	\item\label{enum:001:2} $ \norm{xy}\leq\norm x\norm y $ for all $ x,y\in G_\infty $.
	\item\label{enum:001:3} $ \norm{k_1xk_2}=\norm x $ for all $ k_1,k_2\in K_\infty $ and $ x\in G_\infty $.
\end{enumerate}

\vskip .2in

Any two norms $ \norm\spacedcdot_1 $ and $ \norm\spacedcdot_2 $ on $ G_\infty $ are equivalent in the sense that there exist $ M,r\in\bbR_{>0} $ such that $ \norm x_1\leq M\norm x_2^r $ for all
$ x\in G_\infty $. In the following, we fix a norm $ \norm\spacedcdot $ on $ G_\infty $.

\vskip .2in

By \cite[Lemma 1.10]{casselman}, we have the following lemma.

\vskip .2in

\begin{Lem}\label{lem:099}
	Let $ \Gamma $ be a discrete subgroup of $ G_\infty $. Then, there exist $ N,M\in\bbR_{>0} $ such that 
	\[ \sum_{\gamma\in\Gamma}\norm{\gamma x}^{-N}\leq M,\qquad x\in G_\infty. \]
\end{Lem}

\vskip .2in 

Next, following Casselman \cite{casselman}, for every discrete subgroup $ \Gamma $ of $ G_\infty $ we define the function $ \norm\spacedcdot_{\Gamma\backslash G_\infty}:G_\infty\to\bbR_{\geq1} $,

\[ \norm x_{\Gamma\backslash G_\infty}:=\inf_{\gamma\in\Gamma}\norm{\gamma x}. \]

\vskip .2in

We note two simple consequences of \ref{enum:001:1} and \ref{enum:001:2} in the following lemma.

\vskip .2in

\begin{Lem}\label{lem:005}
	Let $ \Gamma'\subseteq\Gamma $ be discrete subgroups of $ G_\infty $. Let $ S $ be a complete set of left coset representatives for $ \Gamma' $ in $ \Gamma $, i.e., $ \Gamma=\bigsqcup_{\delta\in S}\delta\Gamma' $. Then, we have the following:
	\begin{enumerate}[label=\textup{(\arabic*)},leftmargin=*,align=left]
		\item\label{lem:005:1} $ \norm x_{\Gamma\backslash G_\infty}\leq\norm x_{\Gamma'\backslash G_\infty}\leq\left(\sup_{\delta\in S}\norm\delta\right)\norm x_{\Gamma\backslash G_\infty} $ for all $ x\in G_\infty $.
		\item\label{lem:005:2} $ \norm x_{g\Gamma g^{-1}\backslash G_\infty}\leq\norm g\norm{g^{-1}x}_{\Gamma\backslash G_\infty} $ for all $ x,g\in G_\infty $.
	\end{enumerate}
\end{Lem}

\vskip .2in

Still following Casselman \cite{casselman}, we say that a smooth function $ f:G_\infty\to\bbC $ is of uniformly moderate growth if there exists $ r\in\bbR_{>0} $ such that for every left-invariant differential operator $ u\in U(\frakg_\infty) $ there exists $ M_u\in\bbR_{>0} $ such that
\begin{equation}\label{eq:078}
\abs{(uf)(x)}\leq M_u\norm x^r,\qquad x\in G_\infty.
\end{equation}
Similarly, we say that a continuous representation $ \pi $ of $ G_\infty $ on a locally convex topological vector space $ V $ is of moderate growth if for every continuous seminorm $ \rho $ on $ V $ there exist $ r\in\bbR_{>0} $ and a continuous seminorm $ \nu $ on $ V $ such that
\begin{equation}\label{eq:082}
\norm{\pi(x)v}_\rho\leq \norm x^r\norm v_\nu,\qquad x\in G_\infty,\ v\in V. 
\end{equation}
Here and throughout the paper, all vector spaces are assumed to be complex. All locally convex topological vector spaces are assumed to be Hausdorff.

\vskip .2in

Next, let $ \bbA $ (resp., $ \bbA_f $) be the ring of adeles (resp., of finite adeles) of $ k $. We recall that there exists a finite set $ S_0\subseteq V_f $ such that for all places $ v\in V_f\setminus S_0 $, $ G $ is defined over $ \calO_v $ and $ G(\calO_v) $ is a hyperspecial maximal compact subgroup of $ G(k_v) $ \cite[3.9.1]{tits}. For $ v\in V_f\setminus S_0 $, let $ K_v:=G(\calO_v) $.
We have
\[ G(\bbA_f)={\prod_{v\in V_f}}'G(k_v), \]
where $ \prod_{v\in V_f}' $ is the restricted product over all $ v\in V_f $ with respect to the subgroups $ K_v $ of $ G(k_v) $, $ v\in V_f\setminus S_0 $. We identify the group $ G(\bbA_f) $ (resp., $ G_\infty $) with its image under the canonical inclusion into
\[ G(\bbA)=G_\infty\times G(\bbA_f). \]

\vskip .2in

The group $ G(k) $ embeds diagonally into $ G(\bbA) $ as a discrete subgroup of finite covolume. It also embeds diagonally into $ G_\infty $ and $ G(\bbA_f) $. 
Next, by intersecting the group $ G(k) $ (embedded diagonally into $ G(\bbA_f) $) with an open compact subgroup $ L $ of $ G(\bbA_f) $, we obtain a subgroup $ \Gamma_L $ of $ G(k) $ which is called a congruence subgroup of $ G(k) $. For simplicity, the images of $ \Gamma_L $ under the diagonal embeddings $ G(k)\hookrightarrow G_\infty $, $ G(k)\hookrightarrow G(\bbA_f) $ and $ G(k)\hookrightarrow G(\bbA) $ will also be denoted by $ \Gamma_L $. Embedded into  $ G_\infty $, $ \Gamma_L $ is a discrete subgroup of $ G_\infty $. We write 
\begin{equation}\label{eq:113}
\Gamma_{c,L}:=\Gamma_{cLc^{-1}},\qquad c\in G(\bbA_f).
\end{equation}

\vskip .2in

Throughout the paper we will use the following lemma.

\vskip .2in

\begin{Lem}
	Let $ L $ and $ L' $ be open compact subgroups of $ G(\bbA_f) $, and let $ c,c'\in G(\bbA_f) $. Then, there exists $ M_{c,L,c',L'}\in\bbR_{>0} $ such that
	\begin{equation}\label{eq:039}
	\norm x_{\Gamma_{c,L}\backslash G_\infty}\leq M_{c,L,c',L'}\norm x_{\Gamma_{c',L'}\backslash G_\infty},\qquad x\in G_\infty.
	\end{equation}
\end{Lem}

\begin{proof}
	Since the subgroups $ cLc^{-1} $ and $ c'L'c'^{-1} $ are mutually commensurable, so are $ \Gamma_{c,L} $ and $ \Gamma_{c',L'} $, hence the claim follows from Lemma  \hyperref[lem:005:1]{\ref*{lem:005}\ref*{lem:005:1}}.
\end{proof}

\vskip .2in

Throughout the paper we use the following notation. We write elements $ x\in G(\bbA) $ in the form $ x=(x_\infty,x_f) $, where $ x_\infty\in G_\infty $ and $ x_f\in G(\bbA_f) $. Moreover, for $ f:G(\bbA)\to\bbC $ and $ c\in G(\bbA_f) $, we define the function $ f_c:G_\infty\to\bbC $,
\[ f_c(x):=f(x,c). \]

\vskip .2in 

Let us recall that a continuous function $ f:G(\bbA)\to\bbC $ is said to be smooth if it has the following two properties:
\begin{enumerate}[label=\textup{(\arabic*)},leftmargin=*,align=left]
	\item For every $ c\in G(\bbA_f) $, $ f_c $ is a smooth function $ G_\infty\to\bbC $.
	\item For every $ x\in G_\infty $, the function $ f(x,\spacedcdot):G(\bbA_f)\to\bbC $ is locally constant.
\end{enumerate}
We will denote the space of smooth functions $ G(\bbA)\to\bbC $ by $ C^\infty(G(\bbA)) $. For $ u\in U(\frakg_\infty) $ and $ f\in C^\infty(G(\bbA)) $, we define the function $ uf:G(\bbA)\to\bbC $ by
\[ (uf)_c:=uf_c,\qquad c\in G(\bbA_f). \]

\vskip .2in

Next, we define the following function spaces:
\begin{align*}
	C(G(\bbA)):=\big\{&f:G(\bbA)\to\bbC:f\text{ is continuous}\big\},\\
	C_c(G(\bbA)):=\big\{&f\in C(G(\bbA)):\supp f\text{ is a compact subset of }G(\bbA)\big\},\\
	C^\infty(G(k)\backslash G(\bbA)):=\big\{&f\in C^\infty(G(\bbA)):f(\delta\spacedcdot)=f\text{ for all }\delta\in G(k)\big\},\\
	C(G(k)\backslash G(\bbA)):=\big\{&f\in C(G(\bbA)):f(\delta\spacedcdot)=f\text{ for all }\delta\in G(k)\big\},\\
	C_c(G(k)\backslash G(\bbA)):=\big\{&f\in C(G(k)\backslash G(\bbA)):\supp f\text{ is a compact subset of }G(k)\backslash G(\bbA)\big\}.
\end{align*}
If $ X $ is one if these spaces and $ L $ is an open compact subgroup of $ G(\bbA_f) $, we denote
\[ X^L:=\left\{f\in X:f(\spacedcdot(1_{G_\infty},l))=f\text{ for all }l\in L\right\}. \]
Analogously to the above definitions of function spaces on $ G(\bbA) $, we define the space $ C_c(G(\bbA_f)) $ as well as the spaces $ C(G_\infty) $, $ C^\infty(G_\infty) $, $ C_c(G_\infty) $, $ C_c^\infty(G_\infty) $, $ C(\Gamma\backslash G_\infty) $, $ C^\infty(\Gamma\backslash G_\infty) $ and $ C_c(\Gamma\backslash G_\infty) $ of functions $ G_\infty\to\bbC $, where $ \Gamma $ is a discrete subgroup of $ G_\infty $. Finally, we define $ C_c^\infty(G(\bbA_f)) $ to be the subspace of locally constant functions in $ C_c(G(\bbA_f)) $. We have
\begin{equation}\label{eq:114}
  C_c^\infty(G(\bbA_f))=\bigcup_{L}\sum_{c\in G(\bbA_f)}\bbC\mathbbm1_{cL}=\bigcup_{L}\sum_{c\in G(\bbA_f)}\bbC\mathbbm1_{LcL}, 
\end{equation}
where $ L $ goes over all open compact subgroups of $ G(\bbA_f) $, and $ \mathbbm1_A $ denotes the characteristic function of $ A\subseteq G(\bbA_f) $.

\vskip .2in

Let us fix Haar measures $ dx_\infty $ on $ G_\infty $ and $ dx_f $ on $ G(\bbA_f) $. The measure $ dx:=dx_\infty\times dx_f $ on $ G(\bbA) $ given by the formula
\begin{equation}\label{eq:110}
\int_{G(\bbA)}f(x)\,dx=\int_{G_\infty}\int_{G(\bbA_f)}f(x_\infty,x_f)\,dx_f\,dx_\infty,\quad f\in C_c(G(\bbA)),
\end{equation}
is a Haar measure on $ G(\bbA) $.

\vskip .2in

For $ \calG\in\left\{G_\infty,G(\bbA_f),G(\bbA)\right\} $, the space $ C_c(\calG) $ is a complex associative algebra under the convolution
\begin{equation}\label{eq:070}
(\varphi_1*\varphi_2)(x):=\int_{\calG}\varphi_1\left(xy^{-1}\right)\varphi_2(y)\,dy=\int_{\calG}\varphi_1(y)\,\varphi_2\left(y^{-1}x\right)\,dy, 
\end{equation}
where $ \varphi_1,\varphi_2\in C_c(\calG) $ and $ x\in\calG $.
If $ \pi $ is a continuous representation of $ \calG $ on a complete locally convex topological vector space $ V $, then $ V $ is a left $ C_c(\calG) $-module under the action
\begin{equation}\label{eq:085}
\pi(f)v:=\int_{\calG}f(x)\,\pi(x)v\,dx,\qquad f\in C_c(\calG),\ v\in V,
\end{equation}
(see \cite[\S2]{hc}), where the integral is to be interpreted as a Gelfand-Pettis integral, i.e., its value is determined by the condition
\[ \scal T{\pi(f)v}=\int_{\calG}f(x)\,\scal T{\pi(x)v}\,dx \]
for all continuous linear functionals $ T:V\to\bbC $ (see Theorem \ref{thm:148}).
In the case when $ \calG=G_\infty $, the G\aa rding subspace $ V_{G_\infty\textup{-G\aa rding}} $ of $ V $ is standardly defined by
\begin{equation}\label{eq:086}
V_{G_\infty\textup{-G\aa rding}}:=\mathrm{span}_\bbC\left\{\pi(\psi)v:\psi\in C_c^\infty(G_\infty),\ v\in V\right\}.
\end{equation}
In the case when $ \calG=G(\bbA) $, we define the G\aa rding subspace $ V_{G(\bbA)\textup{-G\aa rding}} $ of $ V $ by analogy as follows:
\begin{equation}\label{eq:087}
V_{G(\bbA)\textup{-G\aa rding}}:=\mathrm{span}_\bbC\left\{\pi(\alpha)v:\alpha\in C_c^\infty(G_\infty)\otimes C_c^\infty(G(\bbA_f)),\ v\in V\right\}.
\end{equation}
In both cases, the G\aa rding subspace is dense in $ V $: Let $ (\psi_n)_{n\in\bbZ_{>0}}\subseteq C_c^\infty(G_\infty) $ be an approximation of unity, i.e., such that $ \psi_n\geq0 $, $ \int_{G_\infty}\psi_n(x)\,dx=1 $, $ \supp\psi_n\searrow\left\{1_{G_\infty}\right\} $, and $ \left\{\supp\psi_n:n\in\bbZ_{>0}\right\} $ is a neighborhood basis of $ 1_{G_\infty} $ in $ G_\infty $. Let $ (L_n)_{n\in\bbZ_{>0}} $ be a decreasing sequence of open compact subgroups of $ G(\bbA_f) $ that constitute a neighborhood basis of $ 1_{G(\bbA_f)} $ in $ G(\bbA_f) $. Then, $ \pi(\psi_n)v\to v $ (resp., $ \pi\left(\psi_n\otimes\frac{\mathbbm1_{L_n}}{\vol L_n}\right)v\to v $) for all $ v\in V $. 

\vskip .2in

Next, for functions $ \varphi:G(\bbA)\to\bbC $ and $ f:G_\infty\to\bbC $, we define the Poincar\'e series
\[ P_{G(k)}\varphi:=\sum_{\delta\in G(k)}\varphi(\delta\spacedcdot)\qquad\text{and}\qquad P_\Gamma f:=\sum_{\gamma\in\Gamma}f(\gamma\spacedcdot), \]
where $ \Gamma $ is a discrete subgroup of $ G_\infty $.
An invariant Radon measure $ dx $ on $ G(k)\backslash G(\bbA) $ is defined by the condition
\begin{equation}\label{eq:099}
\int_{G(k)\backslash G(\bbA)}\left(P_{G(k)}\varphi\right)(x)\,dx=\int_{G(\bbA)}\varphi(x)\,dx,\qquad \varphi\in C_c(G(\bbA)),
\end{equation}
and an invariant Radon measure $ dx_\infty $ on $ \Gamma\backslash G_\infty $ is defined analogously.

\vskip .2in 
  
We end this section by proving a useful integration formula given in \eqref{eq:024}. Let us fix an open compact subgroup $ L $ of $ G(\bbA_f) $. We recall the following well-known fact \cite{borel1963}: there exists a finite set $ C\subseteq G(\bbA_f) $ such that 
\begin{equation}\label{eq:111}
G(\bbA_f)=\bigsqcup_{c\in C}G(k)cL.
\end{equation}
Let us fix such a set $ C $.

\vskip .2in 

\begin{Lem}
	Let $ f\in C(G(k)\backslash G(\bbA))^L $ be non-negative or such that $ \int_{G(k)\backslash G(\bbA)}\abs{f(x)}\,dx<\infty $. Then, we have
	\begin{equation}\label{eq:024}
	\int_{G(k)\backslash G(\bbA)}f(x)\,dx=\vol(L)\,\sum_{c\in C}\int_{\Gamma_{c,L}\backslash G_\infty}f_c(x_\infty)\,dx_\infty.
	\end{equation}
\end{Lem}

\begin{proof}
	By a standard argument, it suffices to prove \eqref{eq:024} in the case when $ f $ is a non-negative function in $ C_c(G(k)\backslash G(\bbA))^L $. In this case, by the classical theory $ f=P_{G(k)}\varphi $ for some non-negative function $ \varphi\in C_c(G(\bbA)) $. In fact, we can assume that $ \varphi\in C_c(G(\bbA))^L $ (by replacing $ \varphi $ by the function $ \varphi^L(x):=\frac1{\vol L}\int_L\varphi(x_\infty,x_fl)\,dl $). For such $ \varphi $, we have
	\begin{equation}\label{eq:026}
	\begin{aligned}
		\int_{G(k)\backslash G(\bbA)}f(x)\,dx
		&\overset{\phantom{\eqref{eq:099}}}=\int_{G(k)\backslash G(\bbA)}\left(P_{G(k)}\varphi\right)(x)\,dx\\
		&\overset{\eqref{eq:099}}=\int_{G(\bbA)}\varphi(x)\,dx\\
		&\underset{\eqref{eq:111}}{\overset{\eqref{eq:110}}=}\sum_{c\in C}\int_{G(k)cL}\int_{G_\infty}\varphi(x_\infty,x_f)\,dx_\infty\,dx_f\\
		&\overset{\phantom{\eqref{eq:099}}}=\sum_{c\in C}\sum_{\delta\in G(k)/\Gamma_{c,L}}\int_{\delta_fcL}\int_{G_\infty}\varphi(x_\infty,x_f)\,dx_\infty\,dx_f\\
		&\overset{\phantom{\eqref{eq:099}}}=\vol(L)\,\sum_{c\in C}\sum_{\delta\in G(k)/\Gamma_{c,L}}\int_{G_\infty}\varphi(x_\infty,\delta_fc)\,dx_\infty,
	\end{aligned}
	\end{equation}
	where the fourth equality holds by the following elementary equivalence: for all $ c\in G(\bbA_f) $ and $ \delta,\delta'\in G(k) $,
	\[ \delta_fcL=\delta_f'cL\quad\Leftrightarrow\quad\delta\Gamma_{c,L}=\delta'\Gamma_{c,L}. \] 
	Now we have
	\begin{align*}
		\sum_{c\in C}&\int_{\Gamma_{c,L}\backslash G_\infty}f_c(x_\infty)\,dx_\infty\\
		&\overset{\phantom{\eqref{eq:026}}}=\sum_{c\in C}\int_{\Gamma_{c,L}\backslash G_\infty}\left(P_{G(k)}\varphi\right)_c(x_\infty)\,dx_\infty\\
		&\overset{\phantom{\eqref{eq:026}}}=\sum_{c\in C}\int_{\Gamma_{c,L}\backslash G_\infty}\sum_{\delta\in G(k)/\Gamma_{c,L}}\sum_{\gamma\in\Gamma_{c,L}}\varphi\big(\delta_\infty\gamma_\infty x_\infty,\delta_fc\underbrace{c^{-1}\gamma_fc}_{\in L}\big)\,dx_\infty\\
		&\overset{\phantom{\eqref{eq:026}}}=\sum_{c\in C}\sum_{\delta\in G(k)/\Gamma_{c,L}}\int_{\Gamma_{c,L}\backslash G_\infty}P_{\Gamma_{c,L}}\left(\varphi(\delta_\infty\spacedcdot,\delta_fc)\right)(x_\infty)\,dx_\infty\\
		&\overset{\phantom{\eqref{eq:026}}}=\sum_{c\in C}\sum_{\delta\in G(k)/\Gamma_{c,L}}\int_{G_\infty}\varphi(\delta_\infty x_\infty,\delta_fc)\,dx_\infty\\
		&\overset{\phantom{\eqref{eq:026}}}=\sum_{c\in C}\sum_{\delta\in G(k)/\Gamma_{c,L}}\int_{G_\infty}\varphi( x_\infty,\delta_fc)\,dx_\infty\\
		&\overset{\eqref{eq:026}}=\frac1{\vol(L)}\int_{G(k)\backslash G(\bbA)}f(x)\,dx,
	\end{align*}
	where the fifth equality holds by the invariance of the Haar measure $ dx_\infty $ on $ G_\infty $.
	This proves \eqref{eq:024}.
\end{proof}

\section{Adelic reformulation of Casselman's results}\label{sec:104}

Let $ \Gamma $ be a discrete subgroup of $ G_\infty $.
Following Casselman \cite{casselman}, we define the Schwartz space $ \calS(\Gamma\backslash G_\infty) $ to consist of the functions $ f\in C^\infty(\Gamma\backslash G_\infty) $ such that for all left-invariant differential operators $ u\in U(\frakg_\infty) $ and $ n\in\bbZ_{\geq0} $ we have
\[ \norm f_{u,-n,\Gamma}:=\sup_{x\in G_\infty}\abs{(uf)(x)}\norm x_{\Gamma\backslash G_\infty}^n<\infty. \]
We equip $ \calS(\Gamma\backslash G_\infty) $ with the locally convex topology generated by the seminorms $ \norm\spacedcdot_{u,-n,\Gamma} $. By \cite[Proposition 1.8]{casselman}, $ \calS(\Gamma\backslash G_\infty) $ is a Fr\'echet space, and the right regular representation $ r_\Gamma $ of $ G_\infty $ on $ \calS(\Gamma\backslash G_\infty) $ is a smooth representation of moderate growth (see \eqref{eq:082}). 

\vskip .2in

Let $ \calS(\Gamma\backslash G_\infty)' $ be the strong dual of $ \calS(\Gamma\backslash G_\infty) $. We recall that this means that $ \calS(\Gamma\backslash G_\infty)' $ is the space of continuous linear functionals $ \calS(\Gamma\backslash G_\infty)\to\bbC $ equipped with the locally convex topology generated by the seminorms 
\[ \norm T_B:=\sup_{f\in B}\abs{\scal Tf}, \]
where $ B $ goes over all bounded sets in $ \calS(\Gamma\backslash G_\infty) $
(see Definitions \ref{def:017} and \ref{def:018}). By Lemma \ref{lem:019}, $ \calS(\Gamma\backslash G_\infty)' $ is a complete locally convex topological vector space. 

\vskip .2in

To the best of our knowledge, the proof of the following lemma is missing from the literature, so we include a sketch of it here.

\vskip .2in

\begin{Lem}\label{lem:136}
	The contragredient representation $ \left(r_\Gamma',\calS(\Gamma\backslash G_\infty)'\right) $,
	\[ r_\Gamma'(x)T:=T\circ r_\Gamma\left(x\right)^{-1},\qquad x\in G_\infty,\ T\in\calS(\Gamma\backslash G_\infty)', \]
	is a smooth represetation of $ G_\infty $. 
\end{Lem}

\begin{proof}
	By the discussion at the beginning of Section \ref{sec:050}, the continuity of the representation $ r_\Gamma' $ is not an immediate consequence of the continuity of $ r_\Gamma $. On the bright side, to prove it, we only need to show that for every $ T\in\calS(\Gamma\backslash G_\infty)' $ the map $ r_\Gamma'(\spacedcdot)T:G_\infty\to\calS(\Gamma\backslash G_\infty)' $ is continuous (at $ 1_{G_\infty} $) (see \eqref{eq:112}). In other words, we need to prove that for every bounded set $ B $ in $ \calS(\Gamma\backslash G_\infty) $ and for every $ T\in\calS(\Gamma\backslash G_\infty)' $, 
	\[ \norm{r_\Gamma'(x)T-T}_B\xrightarrow{x\to 1_{G_\infty}}0. \]
	By the continuity of $ T $, there exist $ m\in\bbZ_{>0} $, $ u_1,\ldots,u_m\in U(\frakg_\infty) $ and $ n_1,\ldots,n_m\in\bbZ_{\geq0} $ such that
	\[ \norm{r_\Gamma'(x)T-T}_B=\sup_{f\in B}\abs{\scal T{r_\Gamma\left(x^{-1}\right)f-f}}\leq\sup_{f\in B}\sum_{i=1}^{m}\norm{r_\Gamma\left(x^{-1}\right)f-f}_{u_i,-n_i,\Gamma} \]
	for all $ x\in G_\infty $. The right-hand side tends to $ 0 $ when $ x\to1_{G_\infty} $ by a standard estimate that uses the definition of seminorms $ \norm\spacedcdot_{u_i,-n_i,\Gamma} $ and the mean value theorem; we leave the details to the reader.
	
	The smoothness of $ r_\Gamma' $ can now be proved as in the proof of Proposition \hyperref[prop:044:2]{\ref*{prop:044}\ref*{prop:044:2}} (see Section \ref{sec:050}).
\end{proof}

\vskip .2in

Still following Casselman \cite{casselman}, as well as \cite{BJ}, \cite{muicRadHAZU} and \cite{muicMathZ}, we define the space 
\begin{equation}\label{eq:154}
\calA_{umg}(\Gamma\backslash G_\infty):=\left\{f\in C^\infty(\Gamma\backslash G_\infty):f\text{ is of uniformly moderate growth}\right\} 
\end{equation}
(see \eqref{eq:078}), its subspace of $ Z(\frakg_\infty) $-finite vectors (smooth automorphic forms for $ \Gamma $) 
\begin{equation}\label{eq:153}
\calA^\infty(\Gamma\backslash G_\infty):=\left\{f\in \calA_{umg}(\Gamma\backslash G_\infty):\dim_\bbC\left\{zf:z\in Z(\frakg_\infty)\right\}<\infty\right\},
\end{equation}
and its subspace of $ K_\infty $-finite vectors (($ K_\infty $-finite) automorphic forms for $ \Gamma $)
\begin{equation}\label{eq:098}
\calA(\Gamma\backslash G_\infty):=\left\{f\in \calA^\infty(\Gamma\backslash G_\infty):\dim_\bbC\mathrm{span}_\bbC\left\{f(\spacedcdot k):k\in K_\infty\right\}<\infty\right\}.
\end{equation}
The spaces $ \calA_{umg}(\Gamma\backslash G_\infty) $ and $ \calA^\infty(\Gamma\backslash G_\infty) $ are $ G_\infty $-modules under the right translations, and $ \calA(\Gamma\backslash G_\infty) $ is a $ (\frakg_\infty,K_\infty) $-module. All these spaces embed canonically into $ \calS(\Gamma\backslash G_\infty)' $ by identifying $ \varphi\in\calA_{umg}(\Gamma\backslash G_\infty) $ with the linear functional $ T_\varphi\in \calS(\Gamma\backslash G_\infty)' $,
\begin{equation}\label{eq:145}
\scal{T_\varphi}f:=\int_{\Gamma\backslash G_\infty}\varphi(x)f(x)\,dx,\qquad f\in\calS(\Gamma\backslash G_\infty).
\end{equation}

\vskip .2in

Adelic analogues of the spaces \eqref{eq:154}--\eqref{eq:098} can be defined in the following standard way (cf.~\cite[\S4]{BJ}). For an open compact subgroup $ L $ of $ G(\bbA_f) $, we define the space
\begin{equation}\label{eq:054}
	\begin{aligned}
		\calA_{umg}(G(k)\backslash G(\bbA))^L:=\big\{&f\in C^\infty(G(k)\backslash G(\bbA))^L:\\
		&f_c\text{ is of uniformly moderate growth for every }c\in G(\bbA_f)\big\},
	\end{aligned}
\end{equation}
its subspace of $ Z(\frakg_\infty) $-finite vectors
\begin{equation}\label{eq:055}
	\calA^\infty(G(k)\backslash G(\bbA))^L:=\big\{f\in \calA_{umg}(G(k)\backslash G(\bbA))^L:\dim_\bbC\left\{zf:z\in Z(\frakg_\infty)\right\}<\infty\big\},
\end{equation}
and its subspace of $ K_\infty $-finite vectors
\begin{equation}\label{eq:056}
	\begin{aligned}
		\calA(G(k)\backslash G(\bbA))^L:=\big\{&f\in\calA^\infty(G(k)\backslash G(\bbA))^L:\\
		&\dim_\bbC\mathrm{span}_\bbC\left\{f\left(\spacedcdot\left(k,1_{G(\bbA_f)}\right)\right):k\in K_\infty\right\}<\infty\big\}.
	\end{aligned}
\end{equation}

\vskip .2in 

Next, we define an adelic analogue of the space $ \calS(\Gamma\backslash G_\infty) $. Let $ L $ be an open compact subgroup of $ G(\bbA_f) $. We define a vector space $ \calS^L=\calS(G(k)\backslash G(\bbA))^L $ to consist of the functions $ f\in C^\infty(G(k)\backslash G(\bbA))^L $ such that
\[ \norm f_{u,-n,c,L}:=\sup_{x\in G_\infty}\abs{(uf)(x,c)}\norm x^n_{\Gamma_{c,L}\backslash G_\infty}<\infty \]
for all $ u\in U(\frakg_\infty) $, $ n\in\bbZ_{\geq0} $ and $ c\in G(\bbA_f) $, where the group $ \Gamma_{c,L} $ is defined by \eqref{eq:113}.
Equipped with the locally convex topology generated by the seminorms $ \norm\spacedcdot_{u,-n,c,L} $, $ \calS^L $ is, by a standard argument, a complete locally convex topological vector space. Moreover, we will show in Lemma \ref{lem:003} that its topology is generated by a countable family of seminorms, hence it is a Fr\'echet space. Thus, by Lemma \ref{lem:019}, its strong dual $ \left(\calS^L\right)' $ is a complete locally convex topological vector space.

\vskip .2in 

As in Section \ref{sec:023}, let $ C $ be a finite subset of $ G(\bbA_f) $ such that $ G(\bbA_f)=\bigsqcup_{c\in C}G(k)cL $.

\vskip .2in 

\begin{Lem}\label{lem:003}
	 The topology of $ \calS^L $ is generated by the seminorms 
	\begin{equation}\label{eq:004}
	\norm\spacedcdot_{u,-n,c,L},\qquad u\in U(\frakg_\infty),\ n\in\bbZ_{\geq0},\ c\in C.
	\end{equation}
\end{Lem}

\begin{proof}
	Let $ u\in U(\frakg_\infty) $, $ n\in\bbZ_{\geq0} $ and $ c_0\in G(\bbA_f) $. It suffices to show that the seminorm $ \norm\spacedcdot_{u,-n,c_0,L} $ is continuous with respect to the locally convex topology on $ \calS^L $ generated by the seminorms \eqref{eq:004}. By definition of $ C $, there exist $ \delta=(\delta_\infty,\delta_f)\in G(k) $, $ c\in C $ and $ l\in L $ such that $ c_0=\delta_fcl $. The claim now follows from the inequality
	\begin{align*}
	\norm f_{u,-n,c_0,L}&\overset{\phantom{\text{Lemma\,\ref{lem:005}\ref{lem:005:2}}}}=\sup_{x\in G_\infty}\abs{(uf)(x,\delta_fcl)}\norm x^n_{\Gamma_{\delta_fcl,L}\backslash G_\infty}\\
	&\overset{\phantom{\text{Lemma\,\ref{lem:005}\ref{lem:005:2}}}}=\sup_{x\in G_\infty}\abs{(uf)\left(\delta_\infty^{-1}x,c\right)}\norm x^n_{\delta_\infty\Gamma_{c,L}\delta_\infty^{-1}\backslash G_\infty}\\
	&\overset{\text{Lemma\,\hyperref[lem:005:2]{\ref*{lem:005}\ref*{lem:005:2}}}}\leq\norm{\delta_\infty}^n\sup_{x\in G_\infty}\abs{(uf)\left(\delta_\infty^{-1}x,c\right)}\norm{\delta_\infty^{-1}x}^n_{\Gamma_{c,L}\backslash G_\infty}\\
	&\overset{\phantom{\text{Lemma\,\ref{lem:005}\ref{lem:005:2}}}}=\norm{\delta_\infty}^n\norm f_{u,-n,c,L},\qquad\quad f\in\calS^L.\qedhere
	\end{align*}
\end{proof} 

\vskip .2in

To clarify the relation between the real and the adelic situation, we note that for $ C\subseteq G(\bbA_f) $ as above, the direct sum of Fr\'echet spaces $ \bigoplus_{c\in C}\calS(\Gamma_{c,L}\backslash G_\infty) $ is itself a Fr\'echet space. Its topology is generated by the seminorms
\[ (f_a)_{a\in C}\mapsto\norm{f_c}_{u,-n,\Gamma_{c,L}},\qquad u\in U(\frakg_\infty),\ n\in\bbZ_{\geq0},\ c\in C. \]
The representation
\[ \left(\bigoplus_{c\in C}r_{\Gamma_{c,L}},\bigoplus_{c\in C}\calS(\Gamma_{c,L}\backslash G_\infty)\right) \]
is a smooth representation of $ G_\infty $ of moderate growth, being a direct sum of such representations. 

\vskip .2in

\begin{Lem}\label{lem:009}
	\begin{enumerate}[label=\textup{(\arabic*)},leftmargin=*,align=left]
		\item\label{lem:009:1} The representation $ \left(r_L,\calS^L\right) $ of $ G_\infty $ defined by
		\[ r_L(x)f:=f\left(\spacedcdot(x,1_{G(\bbA_f)})\right),\qquad x\in G_\infty,\ f\in\calS^L, \]
		is a smooth Fr\'echet representation of moderate growth.
		\item\label{lem:009:2} The rule $ f\mapsto(f_c)_{c\in C} $ defines an equivalence 
		\[ \Phi:\calS^L\to\bigoplus_{c\in C}\calS(\Gamma_{c,L}\backslash G_\infty) \]
		of representations $ r_L $ and $ \bigoplus_{c\in C}r_{\Gamma_{c,L}} $.
	\end{enumerate}
\end{Lem}

\begin{proof}
	One checks easily that $ f\mapsto (f_c)_{c\in C} $ is a linear isomorphism $ C^\infty(G(k)\backslash G(\bbA))^L\to\bigoplus_{c\in C}C^\infty(\Gamma_{c,L}\backslash G_\infty) $. Its inverse maps $ (f_c)_{c\in C}\in\bigoplus_{c\in C}C^\infty(\Gamma_{c,L}\backslash G_\infty) $ to the function $ f\in C^\infty(G(k)\backslash G(\bbA))^L $ defined by
	\[ f(x,\delta_fcl)=f_c\left(\delta_\infty^{-1}x\right),\qquad x\in G_\infty,\ \delta\in G(k),\ c\in C,\ l\in L.  \]
	That these linear isomorphisms by restriction become isomorphisms of Fr\'echet spaces $ \calS^L $ and $ \bigoplus_{c\in C}\calS(\Gamma_{c,L}\backslash G_\infty) $, follows from Lemma \ref{lem:003} and the fact that
	\begin{equation}\label{eq:008}
	\norm f_{u,-n,c,L}=\norm{f_c}_{u,-n,\Gamma_{c,L}}
	\end{equation}
	for all $ f\in C^\infty(G(k)\backslash G(\bbA))^L $, $ u\in U(\frakg_\infty) $, $ n\in\bbZ_{\geq0} $ and $ c\in G(\bbA_f) $. Finally, one checks easily that by pulling back the representation $ \bigoplus_{c\in C}r_{\Gamma_{c,L}} $ to $ \calS^L $ via the isomorphism $ \Phi $, one obtains the representation $ r_L $. The lemma follows.
\end{proof}

\vskip .2in 

\begin{Lem}\label{lem:022}
	\begin{enumerate}[label=\textup{(\arabic*)},leftmargin=*,align=left]
		\item\label{lem:022:1} The contragredient representation $ \left(r_L',\left(\calS^L\right)'\right) $ is a smooth representation of $ G_\infty $.
		\item\label{lem:022:2} The linear operator $ \Psi:\bigoplus_{c\in C}\calS(\Gamma_{c,L}\backslash G_\infty)'\to\left(\calS^L\right)' $, $ (T_c)_{c\in C}\mapsto T $, where
		\[ \scal Tf:=\sum_{c\in C}\scal{T_c}{f_c},\qquad f\in\calS^L, \]
		is an equivalence of representations $ \bigoplus_{c\in C}r_{\Gamma_{c,L}}' $ and $ r_L' $. 
	\end{enumerate}
\end{Lem}

\begin{proof}
	For every $ c\in C $, let $ \iota_{c,L} $ be the canonical inclusion $ \calS(\Gamma_{c,L}\backslash G_\infty)\hookrightarrow\bigoplus_{a\in C}\calS(\Gamma_{a,L}\backslash G_\infty) $. Using Lemma \hyperref[lem:020:2]{\ref*{lem:020}\ref*{lem:020:2}}, one sees easily that the rule $ T\mapsto \left(T\circ\iota_{c,L}\right)_{c\in C} $ defines an equivalence $ \Lambda $ of representations
	\begin{equation}\label{eq:021}
	\left(\left(\bigoplus_{c\in C}r_{\Gamma_{c,L}}\right)',\left(\bigoplus_{c\in C}\calS(\Gamma_{c,L}\backslash G_\infty)\right)'\right) \text{ and } \left(\bigoplus_{c\in C}r_{\Gamma_{c,L}}',\bigoplus_{c\in C}\calS(\Gamma_{c,L}\backslash G_\infty)'\right).
	\end{equation}
	In turn, by Lemma \hyperref[lem:009:2]{\ref*{lem:009}\ref*{lem:009:2}} the linear operator $ \Phi^*:\left(\bigoplus_{c\in C}\calS(\Gamma_{c,L}\backslash G_\infty)\right)'\to\left(\calS^L\right)' $,
	\[ \Phi^*(T):=T\circ\Phi, \]
	is an equivalence of representations $ \left(\bigoplus_{c\in C}r_{\Gamma_{c,L}}\right)' $ and $ r_L' $. Since $ \Psi=\Phi^*\circ\Lambda^{-1} $, the claim \ref{lem:022:2} follows. The claim \ref{lem:022:1} now follows from \ref{lem:022:2} and Lemma \ref{lem:136}.
\end{proof}

\vskip .2in

\begin{Lem}\label{lem:023}
	\begin{enumerate}[label=\textup{(\arabic*)},leftmargin=*,align=left]
		\item\label{lem:023:1} We have the following embedding of $ G_\infty $-modules: $ \calA_{umg}(G(k)\backslash G(\bbA))^L\hookrightarrow\left(\calS^L\right)' $, $ \varphi\mapsto T_\varphi $, where 
		\begin{equation}\label{eq:027}
		\begin{aligned}
		\scal{T_\varphi}f&:=\int_{G(k)\backslash G(\bbA)}\varphi(x)\,f(x)\,dx=\vol(L)\sum_{c\in C}\int_{\Gamma_{c,L}\backslash G_\infty}\varphi_c(x)f_c(x)\,dx
		\end{aligned}
		\end{equation}
		for all $ \varphi\in\calA_{umg}(G(k)\backslash G(\bbA))^L $ and $ f\in\calS^L $.
		\item\label{lem:023:2} The inverse of $ \Psi $ restricts to the following isomorphisms given by the rule $ \varphi\mapsto ((\vol L)\varphi_c)_{c\in C} $:
		\begin{enumerate}[label=\textup{(\roman*)},leftmargin=*,align=left]
			\item\label{lem:023:2:1} an isomorphism of $ G_\infty $-modules $ \calA_{umg}(G(k)\backslash G(\bbA))^L\to\bigoplus_{c\in C}\calA_{umg}(\Gamma_{c,L}\backslash G_\infty) $
			\item\label{lem:023:2:2} an isomorphism of $ G_\infty $-modules $ \calA^\infty(G(k)\backslash G(\bbA))^L\to\bigoplus_{c\in C}\calA^\infty(\Gamma_{c,L}\backslash G_\infty) $ 
			\item\label{lem:023:2:3} an isomorphism of $ (\frakg_\infty, K_\infty) $-modules $ \calA(G(k)\backslash G(\bbA))^L\to\bigoplus_{c\in C}\calA(\Gamma_{c,L}\backslash G_\infty) $.
		\end{enumerate}
	\end{enumerate}
\end{Lem}

\begin{proof}
	The equality in \eqref{eq:027} follows from \eqref{eq:024}. The rest of the lemma follows easily from \eqref{eq:145} and the definitions of $ \Psi $ and of the function spaces \eqref{eq:154}--\eqref{eq:098} and \eqref{eq:054}--\eqref{eq:056}.
\end{proof}

\vskip .2in 

As a corollary of the above results and Casselman's \cite[Theorem 1.16]{casselman}, in Corollary \ref{cor:073} we determine the G\aa rding subspace of $ \left(\calS^L\right)' $. The following proposition is a special case of \cite[Theorem 1.16]{casselman}.

\begin{Prop}\label{prop:147}
	For every $ c\in G(\bbA_f) $, we have 
	\[ \left(\calS(\Gamma_{c,L}\backslash G_\infty)'\right)_{G_\infty\textup{-G\aa rding}}=\calA_{umg}(\Gamma_{c,L}\backslash G_\infty). \]
\end{Prop}

\vskip .2in

\begin{Cor}\label{cor:073}
	We have
	\[ \left(\left(\calS^L\right)'\right)_{G_\infty\textup{-G\aa rding}}=\calA_{umg}(G(k)\backslash G(\bbA))^L. \]
\end{Cor}

\begin{proof}
	By Proposition \ref{prop:147},
	\[ \left(\bigoplus_{c\in C}\calS(\Gamma_{c,L}\backslash G_\infty)'\right)_{G_\infty\textup{-G\aa rding}}=\bigoplus_{c\in C} \calA_{umg}(\Gamma_{c,L}\backslash G_\infty). \]
	The claim follows from this by applying the equivalence $ \Psi $ and using Lemma \hyperref[lem:023:2:1]{\ref*{lem:022}\ref*{lem:023:2}\ref*{lem:023:2:1}}. 
\end{proof}

\section{The space $ \calS(G(k)\backslash G(\bbA)) $}

In this section we define the space $ \calS=\calS(G(k)\backslash G(\bbA)) $. First, we make the following observation.

\vskip .2in 

\begin{Lem}\label{lem:015}
	Let $ L'\subseteq L $ be open compact subgroups of $ G(\bbA_f) $. Then, $ \calS^L $ is a closed subspace of $ \calS^{L'} $.
\end{Lem}

\begin{proof}
	Let $ u\in U(\frakg_\infty) $, $ n\in\bbZ_{\geq0} $ and $ c\in G(\bbA_f) $. We have
	\[ \norm x_{\Gamma_{c,L}\backslash G_\infty}\leq\norm x_{\Gamma_{c,L'}\backslash G_\infty}\overset{\eqref{eq:039}}\leq M_{c,L',c,L}\norm x_{\Gamma_{c,L}\backslash G_\infty},\qquad x\in G_\infty, \]
	hence
	\begin{equation}\label{eq:135}
	\norm f_{u,-n,c,L}\leq\norm f_{u,-n,c,L'}\leq M_{c,L',c,L}^n\norm f_{u,-n,c,L},\qquad f\in\calS^L.
	\end{equation}
	This shows that $ \calS^L $ is a topological subspace of $ \calS^{L'} $, and it is closed in $ \calS^{L'} $ since it is complete.
\end{proof}

\vskip .2in 

Let us define a vector space
\begin{equation}\label{eq:106}
\calS=\calS(G(k)\backslash G(\bbA)):=\bigcup_{L}\calS^L,
\end{equation}
where $ L $ goes over all open compact subgroups of $ G(\bbA_f) $. We equip $ \calS $ with the inductive limit topology determined by the inclusion maps $ \iota_L:\calS^L\hookrightarrow\calS $, i.e., with the finest locally convex topology on $ \calS $ with respect to which the maps $ \iota_L $ are continuous \cite[Definition 12.2.1]{narici}. Using Lemma \ref{lem:015}, one sees easily that it suffices to require that the maps $ \iota_{L_m} $, $ m\in\bbZ_{>0} $, be continuous for some decreasing sequence $ (L_m)_{m\in\bbZ_{>0}} $ such that $ \left\{L_m:m\in\bbZ_{>0}\right\} $ is a neighborhood basis of $ 1_{G(\bbA_f)} $ in $ G(\bbA_f) $. In other words, $ \calS $ is an LF-space with a defining sequence $ \left(\calS^{L_m}\right)_{m\in\bbZ_{>0}} $ (see Definition \ref{def:045}). In particular, we have the following lemma.

\vskip .2in

\begin{Lem}\label{lem:041}
	\begin{enumerate}[label=\textup{(\arabic*)},leftmargin=*,align=left]
		\item\label{lem:041:1} The space $ \calS $ is a complete locally convex topological vector space.
		\item\label{lem:041:2} For every open compact subgroup $ L $ of $ G(\bbA_f) $, $ \calS^L $ is a closed subspace of $ \calS $.
		\item\label{lem:041:3} A subset $ B $ of $ \calS $ is bounded if and only if there exists an open compact subgroup $ L $ of $ G(\bbA_f) $ such that $ B\subseteq\calS^L $ and $ B $ is bounded in $ \calS^L $.
		\item\label{lem:041:4} Let $ V $ be a locally convex topological vector space. A linear operator $ A:\calS\to V $ is continuous if and only if for every open compact subgroup $ L $ of $ G(\bbA_f) $, the restriction $ A\big|_{\calS^L}:\calS^L\to V $ is continuous.		
		\item\label{lem:041:5} Let $ A:\calS\to\calS $ be a linear operator such that for every open compact subgroup $ L $ of $ G(\bbA_f) $, there exists an open compact subgroup $ L' $ of $ G(\bbA_f) $ such that $ A\left(\calS^L\right)\subseteq\calS^{L'} $. Then, $ A $ is continuous if and only if the restrictions $ A\big|_{\calS^L}:\calS^L\to\calS^{L'} $ are continuous.
	\end{enumerate}
\end{Lem}

\begin{proof}
	The claims \ref{lem:041:1}, \ref{lem:041:2}, \ref{lem:041:3} and \ref{lem:041:4} follow from Lemma \hyperref[lem:006:1]{\ref*{lem:006}\ref*{lem:006:1}}, Lemma \hyperref[lem:006:3]{\ref*{lem:006}\ref*{lem:006:3}}, Lemma \hyperref[lem:006:4]{\ref*{lem:006}\ref*{lem:006:4}} and Lemma \hyperref[lem:006:7]{\ref*{lem:006}\ref*{lem:006:7}}, respectively. The claim \ref{lem:041:5} follows from \ref{lem:041:2} and \ref{lem:041:4}.
\end{proof}

\vskip .2in

The following proposition provides a foundation for doing harmonic analysis on $ \calS $. The proof of its first part is quite tedious, so we postpone it until Section \ref{sec:049}.

\vskip .2in

\begin{Prop}\label{prop:016}
	\begin{enumerate}[label=\textup{(\arabic*)},leftmargin=*,align=left]
		\item\label{prop:016:1} The right regular representation $ (r,\calS) $ of $ G(\bbA) $ is well-defined and continuous, in the sense that the map $ G(\bbA)\times\calS\to\calS $, $ (x,f)\mapsto r(x)f $, is continuous.
		\item\label{prop:016:2} The representation $ \left(r\big|_{G_\infty},\calS\right) $ is a smooth representation of $ G_\infty $. 
	\end{enumerate}
\end{Prop}

\begin{proof}[Proof of Proposition \ref{prop:016}(2)]
	The claim follows from the fact that the representation $ \left(r\big|_{G_\infty},\calS\right) $ is the union of its smooth subrepresentations $ \left(r_L,\calS^L\right) $ (see Lemma \hyperref[lem:009:1]{\ref*{lem:009}\ref*{lem:009:1}} and Lemma \hyperref[lem:041:2]{\ref*{lem:041}\ref*{lem:041:2}}).
\end{proof}

\vskip .2in

Next, let $ \calS'=\calS(G(k)\backslash G(\bbA))' $ be the strong dual of $ \calS $. We recall that this means that $ \calS' $ is the space of continuous linear functionals $ \calS\to\bbC $, and it is equipped with the locally convex topology generated by the seminorms $ \norm\spacedcdot_B:\calS'\to\bbR_{\geq0} $,
\begin{equation}\label{eq:051}
 \norm T_B:=\sup_{f\in B}\abs{\scal Tf}, 
\end{equation}
where $ B $ goes over bounded sets in $ \calS $ (see Definitions \ref{def:017} and \ref{def:018}). By Lemma \ref{lem:019}, $ \calS' $ is a complete locally convex topological vector space. By analogy with Casselman's terminology in \cite{casselman}, we call $ \calS' $ the space of tempered distributions on $ G(k)\backslash G(\bbA) $.

\vskip .2in 

\begin{Rem}\label{rem:054}
	\begin{enumerate}[label=\textup{(\arabic*)},leftmargin=*,align=left]
		\item\label{rem:054:1} By Lemma \hyperref[lem:041:4]{\ref*{lem:041}\ref*{lem:041:4}}, a linear functional $ T:\calS\to\bbC $ is continuous if and only if for every open compact subgroup $ L $ of $ G(\bbA_f) $, the restriction $ T\big|_{\calS^L}:\calS^L\to\bbC $ is continuous.
		\item By Lemma \hyperref[lem:041:3]{\ref*{lem:041}\ref*{lem:041:3}}, the topology on $ \calS' $ is simply the topology of uniform convergence on bounded subsets of subspaces $ \calS^L $. 
	\end{enumerate}
\end{Rem}

\vskip .2in

In the following lemma we describe $ \calS' $ as the projective limit of the spaces $ \left(\calS^L\right)' $. 
This is a special case of the well-known result \cite[Theorem 1.3.5]{gilsdorf} exhibiting a canonical isomorphism, in the category of locally convex topological vector spaces, from the strong dual of a regular inductive limit to the projective limit of strong duals. In our situation a short proof is easily attainable, and we present it for the reader's convenience.

\vskip .2in

Let $ \varprojlim\left(\calS^L\right)' $ denote the projective limit of spaces $ \left(\calS^L\right)' $, taken over the directed set of open compact subgroups $ L $ of $ G(\bbA_f) $ ordered by reverse inclusion, with restriction maps as transition morphisms. Let us equip $ \varprojlim\left(\calS^L\right)' $ with the coarsest (locally convex) topology with respect to which the canonical projections $ p_L $ onto the spaces $ \left(\calS^{L}\right)' $ are continuous.

\vskip .2in

\begin{Lem}\label{lem:107}
	The map $ \calS'\to\varprojlim\left(\calS^L\right)' $
	\begin{equation}\label{eq:053}
	T\mapsto\Big(T\big|_{\calS^L}\Big)_{L},
	\end{equation}
	is an isomorphism of topological vector spaces.
\end{Lem}

\begin{proof}
	By Remark \hyperref[rem:054:1]{\ref*{rem:054}\ref*{rem:054:1}}, the map \eqref{eq:053} is a well-defined linear isomorphism. To prove that it is a homeomorphism, we recall that the topology of $ \calS' $ is generated by the seminorms 
	\[ \norm T_B:=\sup_{f\in B}\abs{\scal Tf},\qquad T\in\calS', \]
	where $ B $ goes over all bounded sets in $ \calS $. On the other hand, by \cite[Example 5.11.6]{narici} the topology of $ \varprojlim\left(\calS^L\right)' $ is generated by the seminorms of the form $ \nu\circ p_{L_0} $, where $ L_0 $ is an open compact subgroup of $ G(\bbA_f) $ and $ \nu $ is a continuous seminorm on $ \left(\calS^{L_0}\right)' $. Thus, the topology of $ \varprojlim\left(\calS^L\right)' $ is generated by the seminorms
	\[ \norm{\left(T_L\right)_L}_{L_0,B}:=\norm{T_{L_0}}_B,\qquad (T_L)_{L}\in\varprojlim\left(\calS^L\right)', \]
	where $ L_0 $ goes over all open compact subgroups of $ G(\bbA_f) $ and $ B $ goes over all bounded subsets of $ \calS^{L_0} $. For all such $ L_0 $ and $ B $ we have
	\[ \norm{\left(T\big|_{\calS^L}\right)_L}_{L_0,B}=\norm{T\big|_{\calS^{L_0}}}_B=\norm T_B,\qquad T\in\calS', \]
	which proves the claim by Lemma \hyperref[lem:041:3]{\ref*{lem:041}\ref*{lem:041:3}}.
\end{proof}

\vskip .2in 

Next, recall that by Proposition \ref{prop:016}\ref{prop:016:1} the right regular representation $ r $ of $ G(\bbA) $ on $ \calS $ is continuous. Its contragredient representation $ \left(r',\calS'\right) $ is defined by the formula
\[ r'(x)T:=T\circ r\left(x^{-1}\right),\qquad x\in G(\bbA),\ T\in\calS'. \]
In Section \ref{sec:050} we prove the following proposition.

\vskip .2in

\begin{Prop}\label{prop:044}
	\begin{enumerate}[label=\textup{(\arabic*)},leftmargin=*,align=left]
		\item\label{prop:044:1} The contragredient representation $ \left(r',\calS'\right) $ is a continuous representation of $ G(\bbA) $.
		\item\label{prop:044:2} The representation $ \left(r'\big|_{G_\infty},\calS'\right) $ is a smooth representation of $ G_\infty $.
	\end{enumerate}
\end{Prop}

\section{Action of Hecke algebra}\label{sec:152}

Let us fix an open compact subgroup $ L $ of $ G(\bbA_f) $. For example, $ L $ could be of the form
\[ L=\prod_{v\in S}L_v\times\prod_{v\in V_f\setminus S}K_v, \]
where $ S $ is a finite subset of $ V_f $ such that $ S_0\subseteq S $, and $ L_v $ is an open compact subgroup of $ G(k_v) $ for every $ v\in S $. Here $ S_0\subseteq V_f $ is defined as in Section \ref{sec:023}; in particular, for every $ v\in V_f\setminus S_0 $, $ K_v:=G(\calO_v) $ is a hyperspecial maximal compact subgroup of $ G(k_v) $.

\vskip .2in

The subalgebra
\begin{align}
\calH\left(G(\bbA_f),L\right)&:=\big\{\varphi\in C_c^\infty(G(\bbA_f)):\varphi(l_1xl_2)=\varphi(x)\text{ for all }l_1,l_2\in L\text{ and }x\in G(\bbA_f)\big\}\nonumber\\
&=\sum_{a\in G(\bbA_f)}\bbC\mathbbm1_{LaL}\label{eq:028}
\end{align}
of the convolution algebra $ C_c(G(\bbA_f)) $ (see \eqref{eq:070}) has a unit element $ e_L:=\frac{\mathbbm1_L}{\vol L} $. 
The algebra
\[ \calH(G(\bbA_f)):=\bigcup_{L}\calH(G(\bbA_f),L)\overset{\eqref{eq:114}}=C_c^\infty(G(\bbA_f)), \]
where $ L $ goes over all open compact subgroups of $ G(\bbA_f) $, is called the Hecke algebra of $ G(\bbA_f) $. The algebra $ \calH(G(\bbA_f)) $ is a special case of the classically defined Hecke algebra of a locally compact totally disconnected group (e.g., see \cite{flath}). In the following, we recall some classical results on the action of $ \calH(G(\bbA_f)) $ in representations of $ G(\bbA) $; these results will provide useful information (Lemmas \hyperref[lem:138:1]{\ref*{lem:138}} and \ref{lem:144}) on representations $ \left(r,\calS\right) $ and $ \left(r',\calS'\right) $.

\vskip .2in

Let $ \pi $ be a continuous representation of $ G(\bbA) $ on a complete locally convex topological vector space $ V $. For $ \calG\in\left\{G(\bbA),G_\infty,G(\bbA_f)\right\} $, the continuous representation $ \pi\big|_\calG $ of $ \calG $ on $ V $ induces a representation of the convolution algebra $ C_c(\calG) $ on $ V $ by continuous linear operators $ \pi(\alpha):V\to V $,
\[ \pi(\alpha)v:=\int_\calG\alpha(y)\,\pi(y)v\,dy,\qquad\alpha\in C_c(\calG),\ v\in V \] 
(see \cite[\S2]{hc}).
In particular, the algebra $ \calH(G(\bbA_f)) $ acts on $ V $ by continuous linear operators $ \pi(\varphi):V\to V $,
\[ \pi(\varphi)v:=\int_{G(\bbA_f)}\varphi(y)\,\pi(1_{G_\infty},y)v\,dy,\qquad\varphi\in\calH(G(\bbA_f)),\ v\in V. \]

\vskip .2in

Let $ L $ be an open compact subgroup of $ G(\bbA_f) $. We define the subspace $ V^L $ of $ L $-invariants in $ V $ by
\[ V^L:=\left\{v\in V:\pi(1_{G_\infty},l)v=v\text{ for all }l\in L\right\}. \]
The space $ V^L $ is obviously a closed $ G_\infty $-invariant subspace of $ V $.

\vskip .2in

\begin{Lem}
	\begin{enumerate}[label=\textup{(\arabic*)},leftmargin=*,align=left]\label{lem:133}
		\item\label{lem:133:1} We have
		\begin{equation}\label{eq:132}
		\pi\big(\calH(G(\bbA_f),L)\big)V\subseteq V^L.
		\end{equation}
		\item\label{lem:133:2} The space $ V^L $ is an $ \calH(G(\bbA_f),L) $-invariant subspace of $ V $, and we have
		\begin{equation}\label{eq:134}
		\pi\left(\frac{\mathbbm1_{LaL}}{\vol L}\right)v=\sum_{b\in LaL/L}\pi(1_{G_\infty},b)v,\qquad a\in G(\bbA_f),\ v\in V^L.
		\end{equation}
		\item\label{lem:133:4} The map $ \spacedcdot^L:V\to V $,
		\[ v^L:=\pi(e_L)v=\frac1{\vol L}\int_L\pi(1_{G_\infty},l)v\,dl, \]
		is a continuous $ G_\infty $-equivariant projection operator onto $ V^L $.
	\end{enumerate}
\end{Lem}

\begin{proof}
	\ref{lem:133:1} We have
	\begin{align*}
	\pi(1_{G_\infty},l)\pi(\varphi)v&\overset{\eqref{eq:149}}=\int_{G(\bbA_f)}\varphi(y)\,\pi(1_{G_\infty},ly)v\,dy\\
	&\overset{\phantom{\eqref{eq:149}}}=\int_{G(\bbA_f)}\varphi\left(l^{-1}y\right)\pi(1_{G_\infty},y)v\,dy\\
	&\overset{\phantom{\eqref{eq:149}}}=\int_{G(\bbA_f)}\varphi(y)\,\pi(1_{G_\infty},y)v\,dy\\
	&\overset{\phantom{\eqref{eq:149}}}=\pi(\varphi)v,\qquad\varphi\in\calH(G(\bbA_f),L),\ v\in V,
	\end{align*}
	where the third equality holds since $ \varphi $ is $ L $-invariant on the left. This proves \eqref{eq:132}. 
	
	\vskip .2in
	
	\ref{lem:133:2} The space $ V^L $ is an $ \calH(G(\bbA_f),L) $-invariant subspace of $ V $ by \ref{lem:133:1}. To prove \eqref{eq:134}, note that for all $ a\in G(\bbA_f) $ and $ v\in V^L $,
	\begin{align*}
	\pi\left(\frac{\mathbbm1_{LaL}}{\vol L}\right)v&=\frac1{\vol L}\int_{LaL}\pi(1_{G_\infty},y)v\,dy\\
	&=\frac1{\vol L}\sum_{b\in LaL/L}\int_{bL}\pi(1_{G_\infty},y)v\,dy\\
	&=\sum_{b\in LaL/L}\pi(1_{G_\infty},b)v.
	\end{align*}
	
	\ref{lem:133:4} We have that $ \pi(e_L)V\subseteq V^L $ by \ref{lem:133:1}, and one checks easily that $ v^L=v $ for all $ v\in V^L $ and, using \eqref{eq:149}, that the (continuous) operator $ \pi(e_L) $ is $ G_\infty $-equivariant.
\end{proof}

\vskip .2in

As our notation suggests, the subspace of $ L $-invariants in $ \calS $ coincides with the space $ \calS^L $ defined in Section \ref{sec:104}. This follows easily from the definitions and \eqref{eq:135}. In particular, we have the following lemma.

\vskip .2in

\begin{Lem}\label{lem:138}
	\begin{enumerate}[label=\textup{(\arabic*)},leftmargin=*,align=left]
		\item\label{lem:138:1} The space $ \calS^L $ is an $ \calH(G(\bbA_f),L) $-module under the action defined by
		\[ r\left(\frac{\mathbbm1_{LaL}}{\vol L}\right)f=\sum_{b\in LaL/L}r(1_{G_\infty},b)f,\qquad a\in G(\bbA_f),\ f\in\calS^L. \]
		\item\label{lem:138:2} The linear operator $ \spacedcdot^L:\calS\to\calS $ defined by the formula
		\begin{equation}\label{eq:136}
		f^L=\frac1{\vol L}\int_Lr(1_{G_\infty},l)f\,dl,\quad f\in\calS,
		\end{equation}
		or equivalently by the formula
		\begin{equation}\label{eq:137}
		f^L(x,c)=\frac1{\vol L}\int_Lf(x,cl)\,dl,\quad f\in\calS,\ x\in G_\infty,\ c\in G(\bbA_f),
		\end{equation}
		is a continuous $ G_\infty $-equivariant projection operator onto $ \calS^L $.
	\end{enumerate}
\end{Lem}

\begin{proof}
	\ref{lem:138:1} This follows from Lemma \hyperref[lem:133:2]{\ref*{lem:133}\ref*{lem:133:2}}.
	
	\vskip .2in
	
	\ref{lem:138:2}	The equality \eqref{eq:137} follows from \eqref{eq:136} by applying the evaluation map $ \mathrm{ev}_{(x,c)}:\calS\to\bbC $, $ f\mapsto f(x,c) $, and using \eqref{eq:150}. To justify this last step, note that $ \mathrm{ev}_{(x,c)} $ is a continuous linear functional on $ \calS $ by Remark \hyperref[rem:054:1]{\ref*{rem:054}\ref*{rem:054:1}} and the following estimate: for every open compact subgroup $ L' $ of $ G(\bbA_f) $,
	\[ \abs{\mathrm{ev}_{(x,c)}(f)}=\abs{f(x,c)}\leq\norm f_{1,0,c,L'},\qquad f\in\calS^{L'}. \]
	All the other claims in \ref{lem:138:2} follow from Lemma \hyperref[lem:133:4]{\ref*{lem:133}\ref*{lem:133:4}}.
\end{proof}

\vskip .2in

By applying Lemma \hyperref[lem:133:4]{\ref*{lem:133}\ref*{lem:133:4}} to the representation $ \left(r',\calS'\right) $, we obtain the following lemma.

\vskip .2in

\begin{Lem}\label{lem:144}
	The linear operator $ \spacedcdot^L:\calS'\to\calS' $ defined by the formula
	\begin{equation}\label{eq:139}
	T^L=\frac1{\vol L}\int_Lr'(1_{G_\infty},l)T\,dl,\qquad T\in\calS',
	\end{equation}
	or equivalently by the formula
	\begin{equation}\label{eq:140}
	\scal{T^L}f=\scal T{f^L},\qquad T\in\calS',\ f\in\calS,
	\end{equation}
	is a continuous $ G_\infty $-equivariant projection operator onto $ \left(\calS'\right)^L $.
\end{Lem}

\begin{proof}
	To derive \eqref{eq:140} from \eqref{eq:139}, we note that for every $ f\in\calS $, the evaluation map $ \mathrm{ev}_f:\calS'\to\bbC $, $ \mathrm{ev}_f(T):=\scal Tf $, is a continuous linear functional on $ \calS' $; namely,
	\[ \abs{\mathrm{ev}_f(T)}=\abs{\scal Tf}=\norm T_{\{f\}},\qquad T\in\calS'. \]
	Thus, by applying $ \mathrm{ev}_f $ to \eqref{eq:139}, we obtain
	\begin{align*}
	\scal{T^L}f&\overset{\eqref{eq:150}}=\frac1{\vol L}\int_L\scal{r'(1_{G_\infty},l)T}f\,dl\\
	&\overset{\phantom{\eqref{eq:150}}}=\frac1{\vol L}\int_L\scal{T}{r\left(1_{G_\infty},l^{-1}\right)f}\,dl\\
	&\overset{\eqref{eq:136}}{\underset{\eqref{eq:150}}=}\scal T{f^L},\qquad T\in\calS',\ f\in\calS.
	\end{align*}
	The rest of the lemma follows from Lemma \hyperref[lem:133:4]{\ref*{lem:133}\ref*{lem:133:4}}.
\end{proof}

\section{G\aa rding subspace of $ \calS(G(k)\backslash G(\bbA))' $}\label{sec:117}

The goal of this section is to determine the G\aa rding subspace of $ \left(r',\calS'\right) $ (see \eqref{eq:087}). We start by studying in more detail the G\aa rding subspace of a general continuous representation $ \pi $ of $ G(\bbA) $ on a complete locally convex topological vector space $ V $. Let such $ \pi $ be fixed, and let us define the following $ G(\bbA) $-invariant subspaces of $ V $:
\begin{align*}
	V_{G_\infty\textup{-smooth}}&:=\left\{v\in V:\pi(\spacedcdot,1_{G(\bbA_f)})v\in C^\infty(G_\infty)\right\},\\
	V_{G(\bbA_f)\textup{-smooth}}&:=\bigcup_{L}V^L,
\end{align*}
where $ L $ goes over all open compact subgroups of $ G(\bbA_f) $, and
\[ V^\infty:=V_{G_\infty\textup{-smooth}}\cap V_{G(\bbA_f)\textup{-smooth}}. \]

\vskip .2in

\begin{Rem}
	Since $ \left(r'\big|_{G_\infty},\calS'\right) $ is a smooth representation of $ G_\infty $ (Proposition \hyperref[prop:044:2]{\ref*{prop:044}\ref*{prop:044:2}}), we have
	\begin{equation}\label{eq:151}
	\left(\calS'\right)^\infty=\left(\calS'\right)_{G(\bbA_f)\text{-smooth}}.
	\end{equation}
\end{Rem}

\vskip .2in

By the discussion after \eqref{eq:087}, the G\aa rding subspace
\begin{equation}\label{eq:143}
V_{G(\bbA)\textup{-G\aa rding}}:=\mathrm{span}_\bbC\left\{\pi(\alpha)v:\alpha\in C_c^\infty(G_\infty)\otimes\calH(G(\bbA_f)),\ v\in V\right\}
\end{equation}
is dense in $ V $. On the other hand, for every open compact subgroup $ L $ of $ G(\bbA_f) $,
\begin{equation}\label{eq:144}
\left(V^L\right)_{G_\infty\textup{-G\aa rding}}:=\mathrm{span}_\bbC\left\{\pi(\psi)v:\psi\in C_c^\infty(G_\infty),\ v\in V^L\right\} 
\end{equation}
is a dense subspace of $ V^L $. A simple relation between the spaces \eqref{eq:143} and \eqref{eq:144} is given in the following lemma.

\vskip .2in

\begin{Lem}
	We have
	\begin{equation}\label{eq:142}
	V_{G(\bbA)\textup{-G\aa rding}}=\bigcup_L\left(V^L\right)_{G_\infty\textup{-G\aa rding}}. 
	\end{equation}
\end{Lem}

\begin{proof}
	It suffices to prove the following equality for every open compact subgroup $ L $ of $ G(\bbA_f) $:
	\[ \left(V^L\right)_{G_\infty\textup{-G\aa rding}}=\mathrm{span}_\bbC\left\{\pi(\alpha)v:\alpha\in C_c^\infty(G_\infty)\otimes\calH(G(\bbA_f),L),\ v\in V\right\}. \]
	\fbox{$ \supseteq $}: For $ \psi\in C_c^\infty(G_\infty) $, $ \varphi\in\calH(G(\bbA_f),L) $ and $ v\in V $, we have
	\[ \pi(\psi\otimes\varphi)v=\pi(\psi)\pi(\varphi)v\overset{\text{Lemma \hyperref[lem:133:1]{\ref*{lem:133}\ref*{lem:133:1}}}}\in\pi(\psi)V^L\subseteq\left(V^L\right)_{G_\infty\textup{-G\aa rding}}. \]
	\fbox{$ \subseteq $}: For $ \psi\in C_c^\infty(G_\infty) $ and $ v\in V^L $, we have
	\[ \pi(\psi)v\overset{\text{Lemma \hyperref[lem:133:4]{\ref*{lem:133}\ref*{lem:133:4}}}}=\pi(\psi)\pi(e_L)v=\pi(\psi\otimes e_L)v, \]
	which implies the claim.
\end{proof}

\vskip .2in

\begin{Cor}\label{cor:141}
	\begin{enumerate}[label=\textup{(\arabic*)},leftmargin=*,align=left]
		\item\label{cor:141:1} $ V_{G(\bbA)\textup{-G\aa rding}}\subseteq V^\infty $.
		\item\label{cor:141:2} The subspace $ V^\infty $ is dense in $ V $.
	\end{enumerate}
\end{Cor}

\begin{proof}
	The claim \ref{cor:141:1} follows from \eqref{eq:142}, using that $ V_{G_\infty\text{-G\aa rding}}\subseteq V_{G_\infty\text{-smooth}} $ by \cite[Lemma 2]{hc}, and \ref{cor:141:2} follows from \ref{cor:141:1} since $ V_{G(\bbA)\textup{-G\aa rding}} $ is dense in $ V $. 
\end{proof}

\vskip .2in

The key to determining the G\aa rding subspace $ \left(\calS'\right)_{G(\bbA)\text{-G\aa rding}} $ is the $ G_\infty $-equivalence $ \left(\calS'\right)^L\to\left(\calS^L\right)' $ of the following proposition.

\vskip .2in

\begin{Prop}\label{lem:054}
	Let $ L $ be an open compact subgroup of $ G(\bbA_f) $.
	Then, we have the following commutative diagram of continuous $ G_\infty $-equivariant linear operators:
	\begin{equation}\label{eq:045}
	\begin{tikzcd}[row sep=huge, column sep = huge]
	\left(r'\big|_{G_\infty},\calS'\right) \arrow[r, twoheadrightarrow, "T\mapsto T|_{\calS^L}"] \arrow[d, twoheadrightarrow, "T\mapsto T^L"']
	& \left(r_L',\left(\calS^L\right)'\right) \arrow[d, equal] \\
	\left(r'\big|_{G_\infty},\left(\calS'\right)^L\right) \arrow[r, rightarrow, "T\mapsto T|_{\calS^L}"', "\simeq"]
	& \left(r_L',\left(\calS^L\right)'\right)
	\end{tikzcd}.
	\end{equation}
	The bottom map is an equivalence of representations of $ G_\infty $.
\end{Prop}

\begin{proof}
	The diagram \eqref{eq:045} commutes, i.e., we have
	\[ T^L\big|_{\calS^L}=T\big|_{\calS^L},\qquad T\in\calS',  \]
	since
	\[ \scal{T^L}f\overset{\eqref{eq:140}}=\scal T{f^L}\overset{\text{Lemma \hyperref[lem:138:2]{\ref*{lem:138}\ref*{lem:138:2}}}}=\scal Tf,\qquad T\in\calS',\ f\in\calS^L. \]

	\vskip .2in

	The left-hand side operator in the diagram is continuous, $ G_\infty $-equivariant and surjective by Lemma \ref{lem:144}.
	The operator on the top is obviously $ G_\infty $-equivariant; it is continuous since
	\[ \norm{T\big|_{\calS^L}}_B=\norm T_B,\qquad T\in\calS', \]
	for every bounded set $ B $ in $ \calS^L $; it is surjective since by a corollary \cite[Corollary 7.3.3]{narici} of the Hahn-Banach theorem every $ T\in(\calS^L)' $ extends to an element of $ \calS' $. 
	
	\vskip .2in 
	
	Since the operator on the bottom of the diagram is a restriction of the one on the top, it is continuous and $ G_\infty $-equivariant; it is also surjective by the commutativity of the diagram; it is injective, i.e., its kernel is trivial, since for every $ T\in\left(\calS'\right)^L $ such that $ T\big|_{\calS^L}=0 $, we have
	\[ \scal Tf\overset{\text{Lemma \ref{lem:144}}}=\scal{T^L}f\overset{\eqref{eq:140}}=\scal{T\big|_{\calS^L}}{f^L}=0,\qquad f\in\calS,  \]
	hence $ T=0 $. Finally, the inverse of the bottom map is continuous since for every bounded set $ B $ in $ \calS $ and all $ T\in\left(\calS'\right)^L $, we have 
	\begin{equation}\label{eq:124}
	\norm T_B=\sup_{f\in B}\abs{\scal Tf}\overset{\text{Lemma \ref{lem:144}}}=\sup_{f\in B}\abs{\scal{T^L}f}\overset{\eqref{eq:140}}=\sup_{f\in B}\abs{\scal T{f^L}}=\norm{T\big|_{\calS^L}}_{B^L},
	\end{equation}
	where $ B^L $ is the image of the bounded set $ B $ in $ \calS $ under the continuous linear operator $ \spacedcdot^L:\calS\to\calS^L $, hence a bounded set in $ \calS^L $.
\end{proof}

\vskip .2in

Let us define the spaces (see \eqref{eq:054}--\eqref{eq:056})
\begin{align}
\calA_{umg}(G(k)\backslash G(\bbA))&:=\bigcup_L\calA_{umg}(G(k)\backslash G(\bbA))^L,\label{eq:119}\\
\calA^\infty(G(k)\backslash G(\bbA))&:=\bigcup_L\calA^\infty(G(k)\backslash G(\bbA))^L,\nonumber\\
\calA(G(k)\backslash G(\bbA))&:=\bigcup_L\calA(G(k)\backslash G(\bbA))^L,\nonumber
\end{align}
where $ L $ goes over all open compact subgroups of $ G(\bbA_f) $. 

\vskip .2in

\begin{Lem}\label{lem:060}
	\begin{enumerate}[label=\textup{(\arabic*)},leftmargin=*,align=left]
		\item\label{lem:060:1} The space $ \calA_{umg}(G(k)\backslash G(\bbA)) $ is a left $ G(\bbA) $-module under the right translations, and the rule $ \varphi\mapsto T_\varphi $, where
		\begin{equation}\label{eq:122}
		\scal{T_\varphi}f:=\int_{G(k)\backslash G(\bbA)}\varphi(x)\,f(x)\,dx,\qquad f\in\calS,
		\end{equation}
		defines an embedding of $ G(\bbA) $-modules $ \calA_{umg}(G(k)\backslash G(\bbA))\hookrightarrow\left(\calS'\right)^\infty $. 
		\item\label{lem:060:1.5} Let $ L $ be an open compact subgroup of $ G(\bbA_f) $. If $ \varphi\in\calA_{umg}(G(k)\backslash G(\bbA))^L $, then $ T_\varphi\in\left(\calS'\right)^L $.
		\item\label{lem:060:2} The space
		$ \calA^\infty(G(k)\backslash G(\bbA)) $
		is a $ G(\bbA) $-submodule of $ \calA_{umg}(G(k)\backslash G(\bbA)) $.
		\item\label{lem:060:3} The space 
		$ \calA(G(k)\backslash G(\bbA)) $
		is a $ (\frakg_\infty,K_\infty)\times G(\bbA_f) $-module.
	\end{enumerate}
\end{Lem}

\begin{proof}
	\ref{lem:060:1} \& \ref{lem:060:1.5} The only non-obvious part of \ref{lem:060:1} is that for every $ \varphi\in\calA_{umg}(G(k)\backslash G(\bbA)) $, $ T_\varphi $ is a well-defined element of $ \left(\calS'\right)^\infty $. Let $ L $ be an open compact subgroup of $ G(\bbA_f) $ such that $ \varphi\in\calA_{umg}(G(k)\backslash G(\bbA))^{L} $. By Remark \hyperref[rem:054:1]{\ref*{rem:054}\ref*{rem:054:1}} and Lemma \ref{lem:015}, to prove that $ T_\varphi\in\calS' $, it suffices to prove that $ T_\varphi\big|_{\calS^{L_0}}\in\left(\calS^{L_0}\right)' $ for every open compact subgroup $ L_0\subseteq L $ of $ G(\bbA_f) $, and this holds by Lemma \hyperref[lem:023:1]{\ref*{lem:023}\ref*{lem:023:1}} since for such $ L_0 $ we have that $ \varphi\in\calA_{umg}(G(k)\backslash G(\bbA))^{L_0} $. Now it is straightforward to prove that $ T_\varphi $ is $ L $-invariant, hence 
	\[ T_\varphi\in\left(\calS'\right)^{L}\overset{\eqref{eq:151}}\subseteq\left(\calS'\right)^\infty. \]
	
	\vskip .2in
	
	\ref{lem:060:2} \& \ref{lem:060:3} By Lemma \hyperref[lem:023:2]{\ref*{lem:023}\ref*{lem:023:2}\ref*{lem:023:2:2}--\ref*{lem:023:2:3}}, $ \calA^\infty(G(k)\backslash G(\bbA)) $ is a $ G_\infty $-module, and $ \calA(G(k)\backslash G(\bbA)) $ is a $ (\frakg_\infty,K_\infty) $-module. Both spaces are obviously invariant under the right translations by elements of $ G(\bbA_f) $. The claims follow.
\end{proof}

\vskip .2in

The following theorem is the main result of this section and an adelic analogue of \cite[Theorem 1.16]{casselman}.

\vskip .2in

\begin{Thm}\label{thm:145}
	\begin{enumerate}[label=\textup{(\arabic*)},leftmargin=*,align=left]
		\item\label{thm:145:1} For every open compact subgroup $ L $ of $ G(\bbA_f) $, we have
		\begin{equation}\label{eq:156}
		\left(\left(\calS'\right)^L\right)_{G_\infty\textup{-G\aa rding}}=\calA_{umg}(G(k)\backslash G(\bbA))^L.
		\end{equation}
		\item\label{thm:145:2} We have
		\[ \left(\calS'\right)_{G(\bbA)\textup{-G\aa rding}}=\calA_{umg}(G(k)\backslash G(\bbA)). \]
	\end{enumerate}
\end{Thm}

\begin{proof}
	\ref{thm:145:1} The image of the space $ \calA_{umg}(G(k)\backslash G(\bbA))^L\subseteq\left(\calS'\right)^L $ under the $ G_\infty $-equivalence $ \left(\calS'\right)^L\to\left(\calS^L\right)' $ of Proposition \ref{lem:054} is clearly the space $ \calA_{umg}(G(k)\backslash G(\bbA))^L\subseteq\left(\calS^L\right)' $, which equals $ \left(\left(\calS^L\right)'\right)_{G_\infty\textup{-G\aa rding}} $ by Corollary \ref{cor:073}. The claim follows.
	
	\vskip .2in
	
	\ref{thm:145:2} This follows from \ref{thm:145:1} by \eqref{eq:142}.
\end{proof}

\section{Closed irreducible admissible subrepresentations of $ \calS(G(k)\backslash G(\bbA))' $}\label{sec:131}

In this section, let $ G $ be a (Zariski) connected semisimple group defined over $ k $. Under this assumption, in Theorem \ref{thm:126} we prove that the closed irreducible admissible subrepresentations of $ \left(r',\calS'\right) $ are the closures in $ \calS' $ of irreducible (admissible) $ (\frakg_\infty,K_\infty)\times G(\bbA_f) $-submodules of $ \calA(G(k)\backslash G(\bbA)) $. We start with a few definitions and basic observations.

\vskip .2in

For a continuous representation $ \pi $ of $ G_\infty $ on a complete locally convex topological vector space $ V $, let us denote by $ V_{K_\infty} $ the $ (\frakg_\infty,K_\infty) $-module of $ G_\infty $-smooth, $ K_\infty $-finite vectors in $ V $. We recall that if $ \pi $ is admissible, then all $ K_\infty $-finite vectors in $ V $ are $ G_\infty $-smooth.

\vskip .2in

\begin{Lem}\label{lem:129}
	Let $ \pi $ be a continuous representation of $ G(\bbA) $ on a complete locally convex topological vector space $ V $.
	Then, the $ (\frakg_\infty,K_\infty)\times G(\bbA_f) $-module
	\[ V^\infty_{K_\infty}:=V_{K_\infty}\cap V^\infty \]
	is dense in $ V $. 
\end{Lem}

\begin{proof}
	For every open compact subgroup $ L $ of $ G(\bbA_f) $, $ \left(V^L\right)_{K_\infty} $ is dense in $ V^L $ by \cite[Lemma 4]{hc}. It follows that the space $ V^\infty_{K_\infty}=\bigcup_L\left(V^L\right)_{K_\infty} $ is dense in $ V^\infty=\bigcup_LV^L $. In turn, the space $ V^\infty $ is dense in $ V $ by Corollary \hyperref[cor:141:2]{\ref*{cor:141}\ref*{cor:141:2}}. The claim follows.
\end{proof}

\vskip .2in

\begin{Def}\label{def:152}
	We say that a $ (\frakg_\infty,K_\infty)\times G(\bbA_f) $-module $ W $ is admissible if it has the following two properties:
	\begin{enumerate}[label=\textup{(\arabic*)},leftmargin=*,align=left]
		\item $ W $ is a smooth $ G(\bbA_f) $-module, i.e.,
		\[ W=\bigcup_LW^L, \]
		where $ L $ goes over all open compact subgroups of $ G(\bbA_f) $ and
		\[ W^L:=\left\{w\in W:l.w=w\text{ for all }l\in L\right\}. \]
		\item For every open compact subgroup $ L $ of $ G(\bbA_f) $, $ W^L $ is an admissible $ (\frakg_\infty,K_\infty) $-module, i.e.,
		\[ \dim_\bbC W^L(\delta)<\infty,\qquad\delta\in\widehat K_\infty. \]
		Here $ \widehat K_\infty $ is the set of equivalence classes of irreducible continuous representations of $ K_\infty $, and $ W^L(\delta) $ is the $ \delta $-isotypic component of $ W^L $.
	\end{enumerate}
\end{Def}

\vskip .2in

\begin{Rem}
	In \cite[\S3]{flath}, admissible $ (\frakg_\infty,K_\infty)\times G(\bbA_f) $-modules are called admissible $ G(\bbA) $-modules.
\end{Rem}

\vskip .2in

\begin{Def}\label{def:129}
	Let $ \pi $ be a continuous representation of $ G(\bbA) $ on a complete locally convex topological vector space $ V $. We say that $ \pi $ is admissible if it satisfies one of the following equivalent conditions:
	\begin{enumerate}[label=\textup{(\arabic*)},leftmargin=*,align=left]
		\item\label{def:129:1} The $ (\frakg_\infty,K_\infty)\times G(\bbA_f) $-module $ V_{K_\infty}^\infty $ is admissible.
		\item\label{def:129:2} For every open compact subgroup $ L $ of $ G(\bbA_f) $, $ V^L $ is an admissible representation of $ G_\infty $.
		\item\label{def:129:3} For every open compact subgroup $ L $ of $ G(\bbA_f) $,
		\[ \dim_\bbC V^L(\delta)<\infty,\qquad\delta\in\widehat K_\infty, \]
		where $ V^L(\delta) $ is the $ \delta $-isotypic component of $ V^L $.
	\end{enumerate}
\end{Def}

\vskip .2in

Let $ \pi $ be a continuous representation of $ G(\bbA) $ on a complete locally convex topological vector space $ V $.
Let $ L $ be an open compact subgroup of $ G(\bbA_f) $. By Lemma \hyperref[lem:133:4]{\ref*{lem:133}\ref*{lem:133:4}}, the linear operator $ \spacedcdot^L:V\to V $,
\begin{equation}\label{eq:152}
v^L:=\frac1{\vol L}\int_L\pi(1_{G_\infty},l)v\,dl,
\end{equation}
is a continuous projection onto $ V^L $. Next, by the classical theory (e.g., see \cite[\S2]{hc}), for every $ \delta\in\widehat K_\infty $ the operator $ E_\delta:=\pi\big|_{K_\infty}\left(d(\delta)\overline\xi_\delta\right):V\to V $,
\[ E_\delta v:=\int_{K_\infty}d(\delta)\,\overline{\xi_\delta(k)}\,\pi(k,1_{G(\bbA_f)})v\,dk, \]
is a continuous projection onto $ V(\delta) $ and restricts to a continuous projection of $ V^L $ onto $ V^L(\delta) $. Here $ d(\delta) $ and $ \xi_\delta $ are the degree and character of $ \delta $, respectively, and $ dk $ is the normalized Haar measure on $ K_\infty $. Thus, the operator $ E_\delta^L:V\to V $,
\[ E_\delta^Lv:=E_\delta v^L=\frac{d(\delta)}{\vol L}\int_{K_\infty}\int_L\overline{\xi_\delta(k)}\,\pi(k,l)v\,dl\,dk, \]
is a continuous projection onto $ V^L(\delta) $. Moreover, we have the following lemma.

\vskip .2in

\begin{Lem}
	For every $ (\frakg_\infty,K_\infty)\times G(\bbA_f) $-submodule $ W $ of $ V_{K_\infty}^\infty $, $ E_\delta^L $ restricts to a projection of $ W $ onto $ W^L(\delta) $.
\end{Lem}

\begin{proof}
	By the above discussion, we only need to prove that $ W $ is invariant under the operator $ E_\delta^L $. If $ v\in W $, then $ v\in V^{L'} $ for some $ L'\subseteq L $, and we have
	\[ v^L\overset{\eqref{eq:152}}=\frac1{\vol L}\sum_{b\in L/L'}\int_{bL'}\pi(1_{G_\infty},l)v\,dl=\frac1{\abs{L:L'}}\sum_{b\in L/L'}\pi(1_{G_\infty},b)v, \]
	which implies that $ v^L\in W $ since $ W $ is $ G(\bbA_f) $-invariant. Thus, $ v^L\in W\cap V^L=W^L $, hence $ E_\delta^Lv=E_\delta v^L\in E_\delta W^L=W^L(\delta) $, where the last equality holds by the classical theory of $ (\frakg_\infty,K_\infty) $-modules. The claim follows.
\end{proof}

\vskip .2in

Next, we recall two classical lemmas to be used in the proof of Theorem \ref{thm:126}.

\vskip .2in

\begin{Lem}\label{lem:128}
	Let $ \left(G_\infty\right)_\circ $ be the identity component of $ G_\infty $. Let $ \pi $ be a continuous representation of $ G_\infty $ on a complete locally convex topological vector space $ V $. If $ v\in V $ is $ G_\infty $-smooth, $ Z(\frakg_\infty) $-finite and $ K_\infty $-finite, then $ \Cl_V\pi(U(\frakg_\infty))v $ is $ \left(G_\infty\right)_\circ $-invariant.
\end{Lem}

\begin{proof}
	This is proved by the argument from the first paragraph of the proof of \cite[Theorem 1]{hc}.
\end{proof}

\vskip .2in

\begin{Lem}\label{lem:130}
	Let
	\[ J:=\left\{\psi\in C_c^\infty(G_\infty):\psi\left(kxk^{-1}\right)=\psi(x)\text{ for all }x\in G_\infty\text{ and }k\in K_\infty\right\}. \]
	Let $ \pi $ be an admissible continuous representation of $ G_\infty $ on a complete locally convex topological vector space $ V $. Then, for every $ v\in V_{K_\infty} $ there exists $ \psi\in J $ such that
	\[ v=\pi(\psi)v. \]
	In particular,
	\begin{equation}\label{eq:155}
	V_{K_\infty}\subseteq V_{G_\infty\textup{-G\aa rding}}.
	\end{equation}
\end{Lem}

\begin{proof}
	This is proved by the argument from the last paragraph of the proof of \cite[Theorem 1]{hc}.
\end{proof}

\vskip .2in

\begin{Thm}\label{thm:126}
	Let $ G $ be a (Zariski) connected semisimple group defined over $ k $. Then, the closed irreducible admissible $ G(\bbA) $-invariant subspaces $ V $ of $ \left(r',\calS'\right) $ stand in one-one correspondence with the irreducible (admissible) $ (\frakg_\infty,K_\infty)\times G(\bbA_f) $-submodules $ W $ of $ \calA(G(k)\backslash G(\bbA)) $, the correspondence $ V\leftrightarrow W $ being
	\[ W=V_{K_\infty}^\infty\qquad\text{and}\qquad V=\Cl_{\calS'}W. \]
\end{Thm}

\begin{proof}
	Let us prove that for every irreducible admissible $ (\frakg_\infty,K_\infty)\times G(\bbA_f) $-submodule $ W $ of $ \calA(G(k)\backslash G(\bbA)) $, $ \overline W:=\Cl_{\calS'}W $ is an irreducible admissible representation of $ G(\bbA) $, and $ \overline W_{K_\infty}^\infty=W $. We will do this in three steps, by proving the following claims:
	\begin{enumerate}[label=\textup{(\arabic*)},leftmargin=*,align=left]
		\item\label{enum:127:1} $ \overline W $ is $ G(\bbA) $-invariant.
		\item\label{enum:127:2} $ \overline W_{K_\infty}^\infty=W $.
		\item\label{enum:127:3} $ \overline W $ is an irreducible representation of $ G(\bbA) $.
	\end{enumerate}
	
	\vskip .2in
	
	\ref{enum:127:1} The space $ \overline W $ is clearly $ K_\infty $-invariant and $ G(\bbA_f) $-invariant, hence it suffices to prove that it is $ \left(G_\infty\right)_\circ $-invariant, where $ \left(G_\infty\right)_\circ $ is the identity component of $ G_\infty $. Since every $ T\in W $ is a $ G_\infty $-smooth, $ Z(\frakg_\infty) $-finite and $ K_\infty $-finite vector in $ \left(r',\calS'\right) $ (see Proposition \hyperref[prop:044:2]{\ref*{prop:044}\ref*{prop:044:2}} and the definition of $ \calA(G(k)\backslash G(\bbA)) $), by Lemma \ref{lem:128} $ \Cl_{\calS'}r'(U(\frakg_\infty))T $ is $ \left(G_\infty\right)_\circ $-invariant. Thus, the space
	\[ W_0:=\sum_{T\in W}\Cl_{\calS'}r'(U(\frakg_\infty))T \]
	is $ \left(G_\infty\right)_\circ $-invariant. Since $ W\subseteq W_0\subseteq\overline W $, we have that $ \overline W=\Cl_{\calS'}W_0 $, hence $ \overline W $ is also $ \left(G_\infty\right)_\circ $-invariant.
	
	\vskip .2in
	
	\ref{enum:127:2} For every $ \delta\in\widehat K_\infty $ and every open compact subgroup $ L $ of $ G(\bbA_f) $, we have
	\[ \overline W^L(\delta)=E_\delta^L\overline W\subseteq\Cl_{\calS'}E_\delta^LW=\Cl_{\calS'}W^L(\delta)=W^L(\delta), \]
	where the set inclusion is valid by the continuity of $ E_\delta^L $, and the last equality holds because $ W^L(\delta) $ is finite-dimensional. It follows that $ \overline W^\infty_{K_\infty}\subseteq W $, and the reverse inclusion is obvious.
	
	\vskip .2in
	
	\ref{enum:127:3} Let $ U $ be a closed $ G(\bbA) $-invariant subspace of $ \overline W $. Then, $ U_{K_\infty}^\infty $ is a $ (\frakg_\infty,K_\infty)\times G(\bbA_f) $-submodule of $ \overline W^\infty_{K_\infty}=W $, hence $ U_{K_\infty}^\infty\in\left\{0,W\right\} $ by the irreducibility of $ W $. Thus, 
	\[ U\overset{\text{Lemma \ref{lem:129}}}=\Cl_{\calS'}U_{K_\infty}^\infty\in\left\{0,\overline W\right\}.\] 
	This proves the claim.
	
	\vskip .2in
	
	Conversely, let $ V $ be a closed irreducible admissible $ G(\bbA) $-invariant subspace of $ \left(r',\calS'\right) $. The $ (\frakg_\infty,K_\infty)\times G(\bbA_f) $-module $ V_{K_\infty}^\infty $ is admissible by Definition \hyperref[def:129:1]{\ref*{def:129}\ref*{def:129:1}} and dense in $ V $ by Lemma \ref{lem:129}. To finish the proof of the theorem, we need to prove the following two claims:
	\begin{enumerate}[resume]
		\item\label{enum:127:4} $ V_{K_\infty}^\infty\subseteq\calA(G(k)\backslash G(\bbA)) $.  
		\item\label{enum:127:5} $ V_{K_\infty}^\infty $ is an irreducible $ (\frakg_\infty,K_\infty)\times G(\bbA_f) $-module.
	\end{enumerate}
 
	\vskip .2in
	
	\eqref{enum:127:4} Let $ L $ be an open compact subgroup of $ G(\bbA_f) $. By Definition \hyperref[def:129:2]{\ref*{def:129}\ref*{def:129:2}}, $ V^L $ is a closed admissible $ G_\infty $-invariant subspace of $ \left(\calS'\right)^L $, hence
	\[ \left(V^L\right)_{K_\infty}\overset{\eqref{eq:155}}\subseteq\left(V^L\right)_{G_\infty\text{-G\aa rding}}\subseteq\left(\left(\calS'\right)^L\right)_{G_\infty\text{-G\aa rding}}\overset{\eqref{eq:156}}=\calA_{umg}(G(k)\backslash G(\bbA))^L. \]
	Since all vectors in $ \left(V^L\right)_{K_\infty} $ are $ K_\infty $-finite by definition and $ Z(\frakg_\infty) $-finite by \cite[Theorem 3.4.1]{wallach1}, it follows that $ (V^L)_{K_\infty}\subseteq\calA(G(k)\backslash G(\bbA))^L $. Thus, $ V_{K_\infty}^\infty\subseteq\calA(G(k)\backslash G(\bbA)) $.  
	 
	\vskip .2in
	
	\eqref{enum:127:5} Let $ W $ be a non-zero $ (\frakg_\infty,K_\infty)\times G(\bbA_f) $-submodule of $ V_{K_\infty}^{\infty} $. Then, one sees as in the proof of \ref{enum:127:1} that $ \Cl_{\calS'}W $ is $ G(\bbA) $-invariant, hence $ \Cl_{\calS'}W=V $ by the irreducibility of $ V $. Thus, for every open compact subgroup $ L $ of $ G(\bbA_f) $ and every $ \delta\in\widehat K_\infty $, we have
	\[ V^L(\delta)=E_\delta^LV=E_\delta^L\Cl_{\calS'}W\subseteq\Cl_{\calS'}E_\delta^LW=\Cl_{\calS'}W^L(\delta)=W^L(\delta), \]
	where the last equality holds because $ W^L(\delta)\subseteq V^L(\delta) $ is finite-dimensional.
	It follows that $ V_{K_\infty}^\infty\subseteq W $, hence $ W=V_{K_\infty}^\infty $. This proves \eqref{enum:127:5}.
\end{proof}

\vskip .2in

\begin{Rem}
	It is well-known that every closed irreducible $ G(\bbA) $-invariant subspace of $ L^2(G(k)\backslash G(\bbA)) $ (under the right regular representation) is of the form
	\[ \Cl_{L^2(G(k)\backslash G(\bbA))}W \]
	for some irreducible (admissible) $ (\frakg_\infty,K_\infty)\times G(\bbA_f) $-submodule $ W $ of
	\[ \calA^2(G(k)\backslash G(\bbA)):=\calA(G(k)\backslash G(\bbA))\cap L^2(G(k)\backslash G(\bbA)) \]
	\cite[\S4.6]{BJ}. The representations $ \Cl_{L^2(G(k)\backslash G(\bbA))}W $ and $ \Cl_{\calS'}W $ are related as follows. One sees easily that $ L^2(G(k)\backslash G(\bbA)) $ maps into $ \calS' $ via the continuous, $ G(\bbA) $-equivariant, injective linear operator 
	\begin{equation}\label{eq:157}
	\varphi\mapsto T_\varphi,
	\end{equation}
	where $ T_\varphi $ is defined by \eqref{eq:122}. By restriction, the rule \eqref{eq:157} defines a continuous, $ G(\bbA) $-equivariant, injective linear operator
	\[ \Cl_{L^2(G(k)\backslash G(\bbA))}W\to\Cl_{\calS'}W \] 
	whose image is dense in $ \Cl_{\calS'}W $.
\end{Rem}

\section{The space $ \calS(G(\bbA)) $}\label{sec:107}

Let us define a vector space
\[ \calS(G(\bbA)):=\calS(G_\infty)\otimes C_c^\infty(G(\bbA_f)), \]
where $ \calS(G_\infty):=\calS\left(\left\{1_{G_\infty}\right\}\backslash G_\infty\right) $. We will equip $ \calS(G(\bbA)) $ with a locally convex topology with respect to which the linear operator $ P_{G(k)}:\calS(G(\bbA))\to\calS $,
\begin{equation}\label{eq:048}
P_{G(k)}f:=\sum_{\delta\in G(k)}f(\delta\spacedcdot),
\end{equation}
is well-defined, continuous and surjective.

\vskip .2in

For every open compact subgroup $ L $ of $ G(\bbA_f) $ and every compact subset $ \Omega $ of $ G(\bbA_f) $ such that $ \Omega L= \Omega $, let us define the following linear subspace of $ \calS(G(\bbA)) $:
\[ \calS(G(\bbA))_{\Omega}^L:=\left\{f\in\calS(G_\infty)\otimes C_c^\infty( G(\bbA_f))^L:\supp f\subseteq G_\infty\times \Omega\right\}. \]
Writing $ \Omega $ in the form $ \Omega=\bigsqcup_{c\in\Lambda}cL $ for some finite subset $ \Lambda $ of $ G(\bbA_f) $, we have
\begin{equation}\label{eq:049}
\calS(G(\bbA))_\Omega^L=\calS(G_\infty)\otimes\sum_{c\in \Lambda}\bbC\mathbbm1_{cL}.
\end{equation}
Let us equip the space $ \calS(G(\bbA))_\Omega^L $ with the locally convex topology generated by the seminorms
\[ \norm f_{u,-n}:=\sup_{(x,c)\in G(\bbA)}\abs{(uf)(x,c)}\norm{x}^n, \]
where $ u\in U(\frakg_\infty) $ and $ n\in\bbZ_{\geq0} $. This turns $ \calS(G(\bbA))_{\Omega}^L $ into a Fr\'echet space that is isomorphic to the direct sum $ \bigoplus_{c\in\Lambda}\calS(G_\infty) $ via the isomorphism suggested by \eqref{eq:049}.

\vskip .2in

Let us equip the space $ \calS(G(\bbA)) $ with the finest locally convex topology with respect to which the inclusion maps $ \iota_{\Omega}^L:\calS(G(\bbA))_\Omega^L\hookrightarrow\calS(G(\bbA)) $ are continuous. One sees easily that it suffices to require that the inclusion maps $ \iota_{\Omega_m}^{L_m} $, $ m\in\bbZ_{>0} $, be continuous for some sequence $ (L_m,\Omega_m)_{m\in\bbZ_{>0}} $ such that $ \Omega_m\nearrow G(\bbA_f) $, $ L_m\searrow \left\{1_{G(\bbA_f)}\right\} $, and $ \left\{L_m:m\in\bbZ_{>0}\right\} $ is a neighborhood basis of $ 1_{G(\bbA_f)} $ in $ G(\bbA_f) $. In other words, $ \calS(G(\bbA)) $ is an LF-space with a defining sequence $ \left(\calS(G(\bbA))^{L_m}_{\Omega_m}\right)_{m\in\bbZ_{>0}} $ (see Definition \ref{def:045}). In particular, by Lemma \hyperref[lem:006:1]{\ref*{lem:006}\ref*{lem:006:1}} $ \calS(G(\bbA)) $ is a complete locally convex topological vector space.

\vskip .2in

The main result of this section is Theorem \ref{thm:043}, in which we prove that the linear operator $ P_{G(k)}:\calS(G(\bbA))\to\calS $ is well-defined, continuous and surjective. The proof relies on the analogous result \cite[Proposition 1.11 and Theorem 2.2]{casselman} for the operators $ P_{\Gamma_{c,L}}:\calS(G_\infty)\to\calS(\Gamma_{c,L}\backslash G_\infty) $, and uses the following lemma.

\vskip .2in

\begin{Lem}\label{lem:047}
	Let $ A:\calS(G(\bbA))\to\calS $ be a linear operator such that for every open compact subgroup $ L $ of $ G(\bbA_f) $ and every compact subset $ \Omega $ of $ G(\bbA_f) $ such that $ \Omega L=\Omega $, there exists an open compact subgroup $ L' $ of $ G(\bbA_f) $ such that $ A\left(\calS(G(\bbA))^L_\Omega\right)\subseteq\calS^{L'} $. Then, $ A $ is continuous if and only if the restrictions $ A\big|_{\calS(G(\bbA))^L_\Omega}:\calS(G(\bbA))^L_\Omega\to\calS^{L'} $ are continuous.
\end{Lem}

\begin{proof}
	Completely analogous to the proof of Lemma \hyperref[lem:041:5]{\ref*{lem:041}\ref*{lem:041:5}}. 
\end{proof}

\vskip .2in

\begin{Thm}\label{thm:043}
	Let $ L $ be an open compact subgroup of $ G(\bbA_f) $, and let $ \Omega $ be a compact subset of $ G(\bbA_f) $ such that $ \Omega L=\Omega $. Then, we have the following:
	\begin{enumerate}[label=\textup{(\arabic*)},leftmargin=*,align=left]
		\item\label{thm:043:1} For every $ f\in\calS(G(\bbA))_\Omega^L $, the series $ P_{G(k)}f $ defined by \eqref{eq:048} converges absolutely and uniformly on compact subsets of $ G(\bbA) $, and $ P_{G(k)}f\in\calS^L $.
		\item\label{thm:043:2} If $ G(k)\Omega=G(\bbA_f) $ (e.g., if $ \Omega=CL $, where $ C $ is as in \eqref{eq:111}), then
		the linear operator
		\begin{equation}\label{eq:006}
		P_{G(k)}:\calS(G(\bbA))_{\Omega}^L\to\calS^L 
		\end{equation}
		is continuous and surjective.
		\item\label{thm:043:3} The linear operator 
		\[ P_{G(k)}:\calS(G(\bbA))\to\calS \] 
		is continuous and surjective.
	\end{enumerate}
\end{Thm}

\begin{proof}
	\ref{thm:043:1} Let $ f\in\calS(G(\bbA))_{\Omega}^L $ and $ c\in G(\bbA_f) $. Let $ \Delta_{c,L}\subseteq G(k) $ be such that $ G(k)=\bigsqcup_{\delta\in\Delta_{c,L}}\delta\Gamma_{c,L} $. We have
	\begin{equation}\label{eq:158}
	\begin{aligned}
	\left(P_{G(k)}f\right)_c&=\sum_{\delta\in G(k)}f(\delta_\infty\spacedcdot,\delta_fc)\\
	&=\sum_{\delta\in \Delta_{c,L}}\sum_{\gamma\in\Gamma_{c,L}}f\big(\delta_\infty\gamma_\infty\spacedcdot,\delta_fc\underbrace{c^{-1}\gamma_fc}_{\in L}\big)\\
	&=\sum_{\delta\in \Delta_{c,L}}P_{\Gamma_{c,L}}\big(f(\delta_\infty\spacedcdot,\delta_fc)\big).
	\end{aligned}
	\end{equation}
	Since $ f(\delta_\infty\spacedcdot,\delta_fc)\equiv0 $ if $ \delta_fc\notin \Omega $, the set $ \Delta_{c,L} $ on the right-hand side of this equality can be replaced by the set
	\[ \Delta_{c,L}':=\left\{\delta\in \Delta_{c,L}:\delta_fc\in \Omega\right\}. \]
	Note that $ \Delta_{c,L}' $ is finite since $ \Omega $ is a union of finitely many left $ L $-cosets, and the cosets $ \delta_fcL $, $ \delta\in \Delta_{c,L} $, are mutually disjoint. Moreover, for every $ \delta\in \Delta_{c,L} $ we have that $ f(\delta_\infty\spacedcdot,\delta_fc)\in\calS(G_\infty) $, hence the series $ P_{\Gamma_{c,L}}\big(f(\delta_\infty\spacedcdot,\delta_fc)\big) $ converges absolutely on $ G_\infty $ to an element of $ \calS(\Gamma_{c,L}\backslash G_\infty) $ by \cite[Proposition 1.11]{casselman}; the convergence is uniform on compact subsets of $ G_\infty $ by estimates similar to the ones in the proofs of \cite[Lemma 1.10 and Proposition 1.11]{casselman}. Combined with the right $ L $-invariance of $ f $, this implies that the series $ P_{G(k)}f $ converges absolutely and uniformly on compact subsets of $ G(\bbA) $; moreover, $ P_{G(k)}f\in C^\infty(G(k)\backslash G(\bbA))^L $ and
	\begin{equation}\label{eq:009}
	\left(P_{G(k)}f\right)_c=\sum_{\delta\in \Delta_{c,L}'}P_{\Gamma_{c,L}}\left(f(\delta_\infty\spacedcdot,\delta_fc)\right)\in\calS(\Gamma_{c,L}\backslash G_\infty)
	\end{equation}
	for every $ c\in G(\bbA_f) $. By \eqref{eq:008}, it follows that $ P_{G(k)}f\in\calS^L $.
	
	\vskip .2in 
	
	\ref{thm:043:2} To prove that the linear operator \eqref{eq:006} is continuous, let $ u\in U(\frakg_\infty) $, $ n\in\bbZ_{\geq0} $ and $ c\in G(\bbA_f) $. By the continuity of left translations $ \calS(G_\infty)\to\calS(G_\infty) $ and of the operators $ P_{\Gamma_{c,L}}:\calS(G_\infty)\to\calS(\Gamma_{c,L}\backslash G_\infty) $ \cite[Proposition 1.11]{casselman}, there exist $ m\in\bbZ_{>0} $, $ u_1,\ldots,u_m\in U(\frakg_\infty) $ and $ n_1,\ldots,n_m\in\bbZ_{\geq0} $ such that
	\begin{equation}\label{eq:007}
	\norm{P_{\Gamma_{c,L}}\left(\varphi(\delta_\infty\spacedcdot)\right)}_{u,-n,\Gamma_{c,L}}\leq\sum_{i=1}^m\norm\varphi_{u_i,-n_i,\left\{1_{G_\infty}\right\}}
	\end{equation}
	for all $ \varphi\in\calS(G_\infty) $ and $ \delta\in \Delta_{c,L}' $. Thus, we have
	\begin{align*}
	\norm{P_{G(k)}f}_{u,-n,c,L}&\underset{\eqref{eq:009}}{\overset{\eqref{eq:008}}\leq}\sum_{\delta\in \Delta_{c,L}'}\norm{P_{\Gamma_{c,L}}\left(f(\delta_\infty\spacedcdot,\delta_fc)\right)}_{u,-n,\Gamma_{c,L}}\\
	&\underset{\phantom{\eqref{eq:008}}}{\overset{\eqref{eq:007}}\leq}\sum_{\delta\in \Delta_{c,L}'}\sum_{i=1}^m\norm{f(\spacedcdot,\delta_fc)}_{u_i,-n_i,\left\{1_{G_\infty}\right\}}\\
	&\underset{\phantom{\eqref{eq:008}}}\leq\abs{\Delta_{c,L}'}\sum_{i=1}^m\norm f_{u_i,-n_i},\qquad f\in \calS(G(\bbA))_\Omega^L.
	\end{align*}
	This estimate proves that the linear operator \eqref{eq:006} is continuous.
	
	\vskip .2in 
	
	Let us prove that the operator \eqref{eq:006} is surjective. Since $ G(\bbA_f)=G(k)\Omega $, there exists a (finite) set $ C\subseteq\Omega $ such that $ G(\bbA_f)=\bigsqcup_{c\in C}G(k)cL $. Let $ \varphi\in\calS^L $. By Lemma \hyperref[lem:009:2]{\ref*{lem:009}\ref*{lem:009:2}}, $ \varphi_c\in\calS(\Gamma_{c,L}\backslash G_\infty) $ for all $ c\in C $. Thus, by \cite[Theorem 2.2]{casselman} for every $ c\in C $ there exists a function $ f_c\in\calS(G_\infty) $ such that $ P_{\Gamma_{c,L}}f_c=\varphi_c $. We define
	\begin{equation}\label{eq:121}
	f:=\sum_{c\in C}f_c\otimes\mathbbm1_{cL}\in\calS(G(\bbA))_\Omega^L.
	\end{equation}
	We will prove that $ P_{G(k)}f=\varphi $. Since both sides of this equality are $ G(k) $-invariant on the left and $ L $-invariant on the right, it suffices to show that
	$ \left(P_{G(k)}f\right)_c=\varphi_{c} $ for all $ c\in C $. We have
	\begin{align*}
	\left(P_{G(k)}f\right)_c&\overset{\eqref{eq:158}}=\sum_{\delta\in \Delta_{c,L}}P_{\Gamma_{c,L}}\big(f(\delta_\infty\spacedcdot,\delta_fc)\big)\\
	&\overset{\eqref{eq:121}}=\sum_{\delta\in\Delta_{c,L}}\sum_{a\in C}P_{\Gamma_{c,L}}\left(f_a\left(\delta_\infty\spacedcdot\right)\right)\mathbbm1_{aL}\left(\delta_fc\right)\\
	&\overset{\phantom{\eqref{eq:121}}}=P_{\Gamma_{c,L}}f_{c}\\
	&\overset{\phantom{\eqref{eq:121}}}=\varphi_{c},\qquad\qquad c\in C,
	\end{align*}	
	where the third equality holds because for $ a,c\in C $ and $ \delta\in G(k) $ the following elementary equivalence holds:
	\[ \mathbbm1_{aL}(\delta_fc)=1\quad\Leftrightarrow\quad a=c\text{ \ and \ }\delta\in \Gamma_{c,L}. \]
	
	\vskip .2in
	
	\ref{thm:043:3} The claim follows from \ref{thm:043:2} by Lemma \ref{lem:047}.
\end{proof}

\section{Poincar\'e series on $ G(\bbA) $}\label{sec:108}

In this section, we apply our results to the Poincar\'e series $ P_{G(k)}\varphi $ of functions $ \varphi\in L^1(G(\bbA)) $. The main results of this section---Propositions \ref{prop:090} and \ref{prop:091}---may be regarded as the adelic version of \cite[Proposition 6.4]{muicRadHAZU}.

\vskip .2in

It is well-known (e.g., see \cite[\S4]{muicMathAnn}) that for every $ \varphi\in L^1(G(\bbA)) $ the Poincar\'e series
\[ P_{G(k)}\varphi:=\sum_{\delta\in G(k)}\varphi(\delta\spacedcdot) \]
converges absolutely almost everywhere on $ G(\bbA) $ and that
\[ \norm{P_{G(k)}\varphi}_{L^1(G(k)\backslash G(\bbA))}\leq\norm\varphi_{L^1(G(\bbA))}. \]
Thus, $ P_{G(k)} $ is a continuous linear operator $ L^1(G(\bbA))\to L^1(G(k)\backslash G(\bbA)) $, and it is $ G(\bbA) $-equivariant with respect to the right regular representations on $ L^1(G(\bbA)) $ and $ L^1(G(k)\backslash G(\bbA)) $. Next, one checks easily that the rule $ \varphi\mapsto T_\varphi $, where $ T_\varphi $ is defined by \eqref{eq:122}, is a continuous, $ G(\bbA) $-equivariant, injective linear operator $ L^1(G(k)\backslash G(\bbA))\to\calS' $. Thus, the composition
\begin{equation}\label{eq:084}
\varphi\mapsto T_{P_{G(k)}\varphi} 
\end{equation} 
is a continuous $ G(\bbA) $-equivariant linear operator $ L^1(G(\bbA))\to\calS' $. This enables us to prove the following proposition.

\vskip .2in 
 
\begin{Prop}\label{prop:090}
	Let $ \varphi\in L^1(G(\bbA))_{G(\bbA)\text{-G\aa rding}} $. Then, we have the following:
	\begin{enumerate}[label=\textup{(\arabic*)},leftmargin=*,align=left]
		\item\label{prop:088:1} The function $ P_{G(k)}\varphi\in L^1(G(k)\backslash G(\bbA)) $ coincides almost everywhere with an element of $ \calA_{umg}(G(k)\backslash G(\bbA)) $.
		\item\label{prop:088:2} If $ \varphi $ is $ Z(\frakg_\infty) $-finite, then the function $ P_{G(k)}\varphi\in L^1(G(k)\backslash G(\bbA)) $ coincides almost everywhere with an element of $ \calA^\infty(G(k)\backslash G(\bbA)) $.
		\item\label{prop:088:3} If $ \varphi $ is $ Z(\frakg_\infty) $-finite and $ K_\infty $-finite on the right, then the function $ P_{G(k)}\varphi\in L^1(G(k)\backslash G(\bbA)) $ coincides almost everywhere with an element of $ \calA(G(k)\backslash G(\bbA)) $.		
	\end{enumerate}
\end{Prop} 

\begin{proof}
	The image of $ \varphi $ under the continuous $ G(\bbA) $-equivariant linear operator \eqref{eq:084} belongs to $ \left(\calS'\right)_{G(\bbA)\text{-G\aa rding}} $. By Theorem \hyperref[thm:145:2]{\ref*{thm:145}\ref*{thm:145:2}}, this means that $ T_{P_{G(k)}\varphi}\in\calA_{umg}(G(k)\backslash G(\bbA)) $. This implies \ref{prop:088:1}, and the other two claims follow analogously.
\end{proof}

\vskip .2in

Next, we turn our attention to the strong dual $ \calS(G(\bbA))' $, which is a complete locally convex topological vector space by Lemma \ref{lem:019}. One sees easily that the space $ L^1(G(\bbA)) $ maps into $ \calS(G(\bbA))' $ via the continuous injective linear operator $ \varphi\mapsto U_\varphi $, where
\begin{equation}\label{eq:109}
\scal{U_\varphi}f:=\int_{G(\bbA)}\varphi(x)\,f(x)\,dx,\qquad f\in\calS(G(\bbA)).
\end{equation}
The following proposition describes the (absolute) convergence in $ \calS(G(\bbA))' $ of the Poincar\'e series $ P_{G(k)}\varphi $, i.e., of the series
\[ \sum_{\delta\in G(k)}U_{l(\delta)\varphi}, \]
where $ l(\delta)\varphi:=\varphi\left(\delta^{-1}\spacedcdot\right) $.

\vskip .2in

\begin{Prop}\label{prop:091}
	Let $ \varphi\in L^1(G(\bbA)) $. Then, the series $ \sum_{\delta\in G(k)}U_{l(\delta)\varphi} $ converges absolutely in $ \calS(G(\bbA))' $, and we have
	\begin{equation}\label{eq:102}
	\scal{\sum_{\delta\in G(k)}U_{l(\delta)\varphi}}f=\scal{T_{P_{G(k)}\varphi}}{P_{G(k)}f},\qquad f\in\calS(G(\bbA)).
	\end{equation}
\end{Prop}

\begin{proof}
	Let $ B $ be a bounded set in $ \calS(G(\bbA)) $. To prove the first claim, we need to prove that
	\[ \sum_{\delta\in G(k)}\norm{U_{l(\delta)\varphi}}_B<\infty. \]
	By Lemma \hyperref[lem:006:4]{\ref*{lem:006}\ref*{lem:006:4}}, there exist an open compact subgroup $ L $ of $ G(\bbA_f) $ and a compact subset $ \Omega=\bigsqcup_{c\in\Lambda}c L $ of $ G(\bbA_f) $ such that $ B $ is a bounded subset of $ \calS(G(\bbA))_\Omega^L $. We note that for all $ c\in\Lambda $ and $ \delta,\delta'\in G(k) $, the following equivalence holds:
	\begin{equation}\label{eq:100}
	\delta_f^{-1}c L=\delta_f'^{-1}c L\quad\Leftrightarrow\quad\Gamma_{c,L}\delta=\Gamma_{c,L}\delta'. 
	\end{equation}
	Next, applying Lemma \ref{lem:099}, let us fix $ N,M\in\bbR_{>0} $ such that
	\begin{equation}\label{eq:101}
	\sum_{\gamma\in\Gamma_{c,L}}\norm{\gamma x}^{-N}\leq M,\qquad x\in G_\infty,\ c\in\Lambda. 
	\end{equation}
	We have
	\begin{align}
		\sum_{\delta\in G(k)}&\norm{U_{l(\delta)\varphi}}_B
		\underset{\eqref{eq:109}}{\overset{\eqref{eq:052}}\leq}\sum_{\delta\in G(k)}\sup_{f\in B}\int_{G_\infty\times\Omega}\abs{\varphi\left(\delta^{-1}x\right)f(x)}\,dx\label{eq:103}\\
		&\underset{\phantom{\eqref{eq:101}}}\leq\sum_{c\in\Lambda}\sum_{\delta\in G(k)}\sup_{f\in B}\int_{G_\infty\times\delta_f^{-1}c L}\abs{\varphi(x)f(\delta x)}\,dx\nonumber\\
		&\overset{\eqref{eq:100}}=\sum_{c\in\Lambda}\sum_{\delta\in\Gamma_{c,L}\backslash G(k)}\sum_{\gamma\in\Gamma_{c, L}}\sup_{f\in B}\int_{G_\infty\times\delta_f^{-1}c L}\abs{\varphi(x)f(\gamma\delta x)}\,dx\nonumber\\
		&\underset{\phantom{\eqref{eq:101}}}\leq\sum_{c\in\Lambda}\sum_{\delta\in\Gamma_{c, L}\backslash G(k)}\sum_{\gamma\in\Gamma_{c,L}}\sup_{f\in B}\int_{G_\infty\times\delta_f^{-1}c L}\abs{\varphi(x)}\norm{\gamma_\infty\delta_\infty x_\infty}^{-N}\norm f_{1,-N}\,dx\nonumber\\
		&\underset{\phantom{\eqref{eq:101}}}=\left(\sup_{f\in B}\norm f_{1,-N}\right)\,\sum_{c\in\Lambda}\sum_{\delta\in\Gamma_{c,L}\backslash G(k)}\int_{G_\infty\times\delta_f^{-1}cL}\abs{\varphi(x)}\sum_{\gamma\in\Gamma_{c,L}}\norm{\gamma_\infty\delta_\infty x_\infty}^{-N}\,dx\nonumber\\
		&\overset{\eqref{eq:101}}\leq\left(\sup_{f\in B}\norm f_{1,-N}\right)M\sum_{c\in\Lambda}\sum_{\delta\in\Gamma_{c,L}\backslash G(k)}\int_{G_\infty\times\delta_f^{-1}c L}\abs{\varphi(x)}\,dx\nonumber\\
		&\overset{\eqref{eq:100}}\leq\left(\sup_{f\in B}\norm f_{1,-N}\right)M\sum_{c\in\Lambda}\int_{G(\bbA)}\abs{\varphi(x)}\,dx\nonumber\\
		&\underset{\phantom{\eqref{eq:101}}}<\infty.\nonumber
	\end{align}
	
	\vskip .2in
	
	To prove \eqref{eq:102}, let $ f\in\calS(G(\bbA)) $. The linear functional $ \mathrm{ev}_f:\calS(G(\bbA))'\to\bbC $, $ U\mapsto\scal Uf $, is continuous since
	\[ \mathrm{ev}_f(U)=\abs{\scal Uf}\overset{\eqref{eq:052}}=\norm U_{\left\{f\right\}},\qquad U\in\calS(G(\bbA))'. \]
	Thus, we have
	\begin{align*}
		\scal{\sum_{\delta\in G(k)}U_{l(\delta)\varphi}}f&\overset{\phantom{\eqref{eq:099}}}=\sum_{\delta\in G(k)}\scal{U_{l(\delta)\varphi}}f\\
		&\underset{\phantom{\eqref{eq:099}}}{\overset{\eqref{eq:109}}=}\sum_{\delta\in G(k)}\int_{G(\bbA)}\varphi\left(\delta^{-1}x\right)f(x)\,dx\\
		&\overset{\phantom{\eqref{eq:099}}}=\int_{G(\bbA)}\left(P_{G(k)}\varphi\right)(x)\,f(x)\,dx\\
		&\overset{\eqref{eq:099}}=\int_{G(k)\backslash G(\bbA)}\sum_{\delta\in G(k)}\left(P_{G(k)}\varphi\right)(x)\,f(\delta x)\,dx\\
		&\overset{\phantom{\eqref{eq:099}}}=\int_{G(k)\backslash G(\bbA)}\left(P_{G(k)}\varphi\right)(x)\,\left(P_{G(k)}f\right)(x)\,dx\\
		&\overset{\phantom{\eqref{eq:099}}}=\scal{T_{P_{G(k)}\varphi}}{P_{G(k)}f},
	\end{align*}
	where the third equality holds by the dominated convergence theorem (its use is justified by the above estimate of the right-hand side of \eqref{eq:103} in the case when $ B=\left\{f\right\} $).
\end{proof}

\section{Proof of Proposition \hyperref[prop:016:1]{\ref*{prop:016}\ref*{prop:016:1}}}\label{sec:049}

This section is devoted to proving Proposition \hyperref[prop:016:1]{\ref*{prop:016}\ref*{prop:016:1}}. The proof is technical and can be skipped on first reading without loss of continuity.

\vskip .2in

\begin{Lem}\label{lem:011}
	Let $ L $ be an open compact subgroup of $ G(\bbA_f) $. Let $ x\in G_\infty $, $ a\in G(\bbA_f) $ and $ f\in\calS^L $. Then:
	\begin{enumerate}[label=\textup{(\arabic*)},leftmargin=*,align=left]
		\item\label{lem:011:1} $ r(x,a)f\in\calS^{aLa^{-1}} $.
		\item\label{lem:011:2} For all $ u\in U(\frakg_\infty) $, $ n\in\bbZ_{\geq0} $ and $ c\in G(\bbA_f) $,
		\[ \norm{r(x,a)f}_{u,-n,c,aLa^{-1}}\leq\norm x^n\norm f_{\Ad\left(x^{-1}\right)u,-n,ca,L}. \] 
	\end{enumerate}
\end{Lem}

\begin{proof}
	Obviously $ r(x,a)f\in C^\infty\left(G(k)\backslash G(\bbA)\right)^{aLa^{-1}} $. Moreover, we have
	\begin{align*}
	\norm{r(x,a)f}_{u,-n,c,aLa^{-1}}
	&=\sup_{y\in G_\infty}\abs{u\left(f_{ca}(\spacedcdot x)\right)(y)}\norm y_{\Gamma_{c,aLa^{-1}}\backslash G_\infty}^n\\
	&=\sup_{y\in G_\infty}\abs{\left(\left(\Ad\left(x^{-1}\right)u\right)f_{ca}\right)(yx)}\norm y^n_{\Gamma_{ca,L}\backslash G_\infty}\\
	&\leq\sup_{y\in G_\infty}\abs{\left(\left(\Ad\left(x^{-1}\right)u\right)f_{ca}\right)(yx)}\norm{yx}^n_{\Gamma_{ca,L}\backslash G_\infty}\norm x^n\\
	&=\norm f_{\Ad\left(x^{-1}\right)u,-n,ca,L}\norm x^n,
	\end{align*}
	where in the second equality we applied the identity
	\[ u\circ r(x)=r(x)\circ\Ad\left(x^{-1}\right)u,\qquad u\in U(\frakg_\infty),\ x\in G_\infty, \]
	where $ r(x):C^\infty(G_\infty)\to C^\infty(G_\infty) $ is the right translation $ \varphi\mapsto \varphi(\spacedcdot x) $.
\end{proof}

\vskip .2in 

\begin{proof}[Proof of Proposition {\hyperref[prop:016:1]{\ref*{prop:016}}}(1)]
	By Lemma \hyperref[lem:011:1]{\ref*{lem:011}\ref*{lem:011:1}} the group $ G(\bbA) $ acts on $ \calS $ by right translations. We need to prove that this action is continuous, i.e, that for every $ (x_0,f_0)\in G(\bbA)\times\calS $,
	\[ \lim_{(y,f)\to(1_{G(\bbA)},f_0)}\left(r(x_0y)f-r(x_0)f_0\right)=0. \]
	Writing the function under the limit sign in the form
	\[ r(x_0)\big(r(y)(f-f_0)+(r(y)f_0-f_0)\big), \]
	we see that it suffices to prove the following three claims:
	\begin{enumerate}[label=\textup{(\arabic*)},leftmargin=*,align=left]
		\item\label{enum:010:1} The linear operator $ r(x):\calS\to\calS $ is continuous for every $ x\in G(\bbA) $.
		\item\label{enum:010:2} The function $ G(\bbA)\times\calS\to\calS $, $ (x,f)\mapsto r(x)f $, is continuous at $ \left(1_{G(\bbA)},0\right) $.		
		\item\label{enum:010:3} The function $ r(\spacedcdot)f:G(\bbA)\to\calS $ is continuous at $ 1_{G(\bbA)} $ for every $ f\in\calS $.
	\end{enumerate}
	
	\vskip .2in 
	
	\ref{enum:010:1} By Lemma \ref{lem:011} the restrictions $ r(x)\big|_{\calS^L}:{\calS^L}\to{\calS^{x_fLx_f^{-1}}} $ are continuous, hence $ r(x) $ is continuous by Lemma \hyperref[lem:041:5]{\ref*{lem:041}\ref*{lem:041:5}}.
	
	\vskip .2in 
	
	\ref{enum:010:2} Let us fix a decreasing sequence $ \left(L_m\right)_{m\in\bbZ_{>0}} $ of open compact subgroups of $ G(\bbA_f) $ such that $ \left\{L_m:m\in\bbZ_{>0}\right\} $ is a neighborhood basis of $ 1_{G(\bbA_f)} $ in $ G(\bbA_f) $ and that
	\begin{equation}\label{eq:012}
	L_{m+1}\subseteq\bigcap_{a\in L_1/L_m}aL_ma^{-1},\qquad m\in\bbZ_{>0}.
	\end{equation}
	We also fix a compact neighborhood $ W $ of $ 1_{G_\infty} $ in $ G_\infty $.
	
	\vskip .2in 
	
	Our first goal is to construct a neighborhood basis of $ 0 $ in $ \calS $. First, we need some notation. Let $ \aconv(A) $ denote the absolutely convex hull of a set $ A\subseteq\calS $ (see Definition \ref{def:042}). Moreover, for every $ m\in\bbZ_{>0} $, let $ \calU_m $ be the family of open balls 
	\[ B_\sigma(\varepsilon):=\left\{f\in\calS^{L_m}:\sigma(f)<\varepsilon\right\}, \]
	where $ \sigma:\calS^{L_m}\to\bbR_{\geq0} $ is a continuous seminorm and $ \varepsilon\in\bbR_{>0} $.
	The family $ \calU_m $ is a neighborhood basis of $ 0 $ in $ \calS^{L_m} $. Thus, applying Lemma \hyperref[lem:006:2b]{\ref*{lem:006}\ref*{lem:006:2b}}, we can define a neighborhood basis $ \calU $ of $ 0 $ in $ \calS $ to consist of the sets
	\[ U_{(\sigma_m),(\varepsilon_m)}:=\aconv\left(\bigcup_{m\in\bbZ_{>0}}B_{\sigma_m}(\varepsilon_m)\right), \]
	where $ (\sigma_m)_{m\in\bbZ_{>0}} $ is a sequence of continuous seminorms $ \sigma_m:\calS^{L_m}\to\bbR_{\geq0} $, and $ (\varepsilon_m)_{m\in\bbZ_{>0}}\subseteq\bbR_{>0} $. A smaller neighborhood basis $ \calU_0 $ of $ 0 $ in $ \calS $ can be obtained by requiring that the seminorms $ \sigma_m $ be finite sums of seminorms
	\[ \norm\spacedcdot_{u,-n,c,L_m},\qquad u\in U(\frakg_\infty),\ n\in\bbZ_{\geq0},\ c\in G(\bbA_f). \]
	
	\vskip .2in
	
	Let $ U=U_{(\sigma_m),(\varepsilon_m)}\in\calU_0 $ be fixed. To prove \ref{enum:010:2}, it suffices to find a set $ U'=U_{(\sigma_m'),(\varepsilon_m')}\in\calU $ such that for all $ (x,a)\in W\times L_1 $, the following implication holds:
	\[ f\in U'\ \Rightarrow\ r(x,a)f\in U, \]
	i.e.,
	\[ f\in\aconv\left(\bigcup_{m\in\bbZ_{>0}}B_{\sigma_m'}\left(\varepsilon_m'\right)\right)\ \Rightarrow\ r(x,a)f\in\aconv\left(\bigcup_{m\in\bbZ_{>0}}B_{\sigma_m}(\varepsilon_m)\right). \]
	We will prove that for a suitable choice of $ \left(\sigma_m'\right) $ and $ \left(\varepsilon_m'\right) $, the following stronger claim holds: for all $ (x,a)\in W\times L_1 $ and $ m\in\bbZ_{>0} $,
	\begin{equation}\label{eq:036}
	f\in B_{\sigma_m'}\left(\varepsilon_m'\right)\ \Rightarrow\ r(x,a)f\in B_{\sigma_{m+1}}(\varepsilon_{m+1}).
	\end{equation}
	Our transition to \eqref{eq:036} is motivated by the following fact: For every choice of $ \sigma_m' $ and $ \varepsilon_m' $, the operators $ r(x,a) $, $ (x,a)\in W\times L_1 $, restrict to operators $ \calS^{L_m}\to\calS^{L_{m+1}} $. Namely, for all $ (x,a)\in W\times L_1 $,
	\[ r(x,a)\calS^{L_m}\overset{\text{Lemma \ref{lem:011}\ref{lem:011:1}}}\subseteq\cal S^{aL_ma^{-1}}\underset{\text{Lemma \ref{lem:015}}}{\overset{\eqref{eq:012}}\subseteq}\cal S^{L_{m+1}}. \]
	
	\vskip .2in
	
	Let us fix $ m\in\bbZ_{>0} $. To finish the proof, we need to find a continuous seminorm $ \sigma_m':\calS^{L_m}\to\bbR_{\geq0} $ and a real number $ \varepsilon_m'>0 $ such that for all $ (x,a)\in W\times L_1 $ and $ f\in\calS^{L_m} $, the following implication holds:
	\begin{equation}\label{eq:038}
	\sigma_m'(f)<\varepsilon_m'\quad\Rightarrow\quad\sigma_{m+1}(r(x,a)f)<\varepsilon_{m+1}. 
	\end{equation}
	Let us suppose without loss of generality that
	\[ \sigma_{m+1}=\norm\spacedcdot_{u,-n,c,L_{m+1}} \]
	for some $ u\in U(\frakg_\infty) $, $ n\in\bbZ_{\geq0} $ and $ c\in G(\bbA_f) $. 
	
	\vskip .2in
	
	Before proceeding, we recall the standard filtration
	\[ \bbC=U^0(\frakg_\infty)\subseteq U^1(\frakg_\infty)\subseteq U^2(\frakg_\infty)\subseteq\ldots \]
	of $ U(\frakg_\infty) $ by finite-dimensional $ G_\infty $-invariant subspaces. Let $ N\in\bbZ_{>0} $ such that $ u\in U^N(\frakg_\infty) $, and let $ \left\{u_i\right\}_{i=1}^d $ be a basis of $ U^N(\frakg_\infty) $. Then,
	\begin{equation}\label{eq:037}
	\Ad\left(x^{-1}\right)u=\sum_{i=1}^d\eta_i(x)u_i,\qquad x\in G_\infty,
	\end{equation}
	for some functions $ \eta_1,\ldots,\eta_d\in C^\infty(G_\infty) $.
	
	\vskip .2in
	
	For all $ (x,a)\in W\times L_1 $ and $ f\in\calS^{L_m} $, we have
	\begin{align*}
	\sigma_{m+1}&(r(x,a)f)=\lVert r(x,a)f\rVert_{u,-n,c,L_{m+1}}\\
	&\overset{\text{Lemma \hyperref[lem:011:2]{\ref*{lem:011}\ref*{lem:011:2}}}}\leq\norm x^n\norm f_{\Ad\left(x^{-1}\right)u,-n,ca,a^{-1}L_{m+1}a}\\
	&\underset{\text{\phantom{Lemma \ref{lem:011}\ref{lem:011:2}}}}{\overset{\eqref{eq:039},\,\eqref{eq:037}}\leq} \norm x^n M_{c,L_{m+1},ca,L_m}^n\norm f_{\sum_{i=1}^d\eta_i(x)u_i,-n,ca,L_m}\\
	&\overset{\phantom{\text{Lemma \ref{lem:011}\ref{lem:011:2}}}}\leq\sum_{i=1}^d\left(\max_{x\in W}\norm x^n\abs{\eta_i(x)}\right)\max_{b\in L_1/L_{m}}M_{c,L_{m+1},cb,L_m}^n\norm f_{u_i,-n,cb,L_m}\\
	&\overset{\phantom{\text{Lemma \ref{lem:011}\ref{lem:011:2}}}}{=:}\sigma_m'(f).
	\end{align*}
	The seminorm $ \sigma_m' $ defined in this way and $ \varepsilon_m':=\varepsilon_{m+1} $ clearly satisfy \eqref{eq:038}. This finishes the proof of \ref{enum:010:2}.
	
	\vskip .2in
	
	\ref{enum:010:3} Fix $ f\in\calS $, and let $ L $ be an open compact subgroup of $ G(\bbA_f) $ such that $f\in\calS^L $. For every $ (x,a)\in G_\infty\times L $, $ r(x,a)f=r(x,1_{G(\bbA_f)})f\in\calS^L $ by Lemma \ref{lem:011}, and for all $ u\in U(\frakg_\infty) $, $ n\in\bbZ_{\geq0} $ and $ c\in G(\bbA_f) $ we have
	\[ \norm{r(x,a)f-f}_{u,-n,c,L}=\norm{r_L(x)f_{c}-f_{c}}_{u,-n,\Gamma_{c,L}}\xrightarrow{(x,a)\to(1_{G_\infty},1_{G(\bbA_f)})}0 \]
	by the continuity of the representation $ r_L $. This proves \ref{enum:010:3}.
\end{proof}

\section{Proof of Proposition \hyperref[prop:044:1]{\ref*{prop:044}}}\label{sec:050}

Proposition \hyperref[prop:044:1]{\ref*{prop:044}\ref*{prop:044:1}} states that $ \left(r',\calS'\right) $ is a continuous represenation of $ G(\bbA) $. We note that this is not immediately obvious:
in general, if $ \pi $ is a continuous representation of a locally compact Hausdorff group $ \calG $ on an LF-space $ V $, its contragredient representation $ \pi' $ on the strong dual $ V' $ need not be continuous (simple counterexamples are regular representations of $ \calG $ on $ L^1(\calG) $, where $ \calG $ is a Lie group of positive dimension \cite[\S4.1.2, p.~224]{warner}). 
However, the closed subrepresentation
\begin{align}
V^\vee&:=\left\{T\in V':\pi'\left(\spacedcdot\right)T:\calG\to V' \text{ is continuous}\right\}\label{eq:112}\\
&=\Cl_{V'}\mathrm{span}_\bbC\left\{\pi'(\varphi)T:\varphi\in C_c(\calG),\ T\in V'\right\}\nonumber
\end{align}
of $ \pi' $ is continuous \cite[4.1.2, pp.~223-224]{warner}.

\begin{proof}[Proof of Proposition {\hyperref[prop:044:1]{\ref*{prop:044}}}]
	\ref{prop:044:1} We will prove that $ \calS^\vee=\calS' $ or equivalently that the function $ r'(\spacedcdot)T:G(\bbA)\to\calS' $ is continuous for every $ T\in\calS' $. Clearly, it suffices to prove that each of these functions is continuous at $ 1_{G(\bbA)} $. 
	
	\vskip .2in
	
	Let $ T\in\calS' $. We need to prove that for every bounded set $ B $ in $ \calS $,
	\[ \norm{r'(x,a)T-T}_B\xrightarrow{(x,a)\to1_{G(\bbA)}}0. \]
	By Lemma \hyperref[lem:041:3]{\ref*{lem:041}\ref*{lem:041:3}} there exists an open compact subgroup $ L $ of $ G(\bbA_f) $ such that $ B $ is a bounded subset of $ \calS^L $. Note that $ r(x,a)f=r(x,1_{G(\bbA_f)})f $ for $ (x,a)\in G_\infty\times L $ and $ f\in B $, hence we have, for all $ (x,a)\in G_\infty\times L $,
	\begin{align*}
	\norm{r'(x,a)T-T}_B&=\sup_{f\in B}\abs{\scal{\vphantom{\Big(}r'(x,a)T-T}f}\\
	&=\sup_{f\in B}\abs{\scal{\vphantom{\Big(}T}{r\left(x^{-1},a^{-1}\right)f-f}}\\
	&=\sup_{f\in B}\abs{\scal{\vphantom{\Big(}T\big|_{\calS^L}}{r_L\left(x^{-1}\right)f-f}}\\
	&=\sup_{f\in B}\abs{\scal{r_L'(x)\left(\vphantom{\Big(}T\big|_{\calS^L}\right)-T\big|_{\calS^L}}f}\\
	&=\norm{r_L'(x)\left(\vphantom{\Big(}T\big|_{\calS^L}\right)-T\big|_{\calS^L}}_{B},
	\end{align*}
	and the right-hand side tends to $ 0 $ as $ (x,a)\to\left(1_{G_\infty},1_{G(\bbA_f)}\right) $ by the continuity of the representation $ r_L' $ (Lemma \ref{lem:022}\ref{lem:022:1}). This proves the claim.
	
	\vskip .2in
	
	\ref{prop:044:2} 
	Let $ \left(\calS'\right)_{G_\infty\textup{-smooth}} $ denote the space of $ G_\infty $-smooth vectors in $ \calS' $, i.e., the space of all $ T\in\calS' $ such that the map $ G_\infty\to\calS' $, $ x\mapsto r'(x,1_{G(\bbA_f)})T $, is smooth. We need to show that $ \left(\calS'\right)_{G_\infty\textup{-smooth}}=\calS' $.
	
	\vskip .2in

	Let $ u\mapsto u^\# $ be the unique $ \bbC $-linear anti-automorphism of the algebra $ U(\frakg_\infty) $ such that $ X^\#=-X $ for all $ X\in\frakg_\infty $. 
	One sees easily that
	\begin{equation}\label{eq:061}
	\scal{r'(u)T}f=\scal T{r\left(u^\#\right)f}
	\end{equation}
	for all $ u\in U(\frakg_\infty) $, $ T\in\left(\calS'\right)_{G_\infty\textup{-smooth}} $ and $ f\in\calS $.
	
	\vskip .2in
	
	By the argument from the proof of \cite[Lemma 1.6.4(1)]{wallach1} (of course, one needs to work with nets instead of sequences),
	the space $ \left(\calS'\right)_{G_\infty\textup{-smooth}} $ is a complete  topological vector space when equipped with the locally convex topology $ \tau $ generated by the seminorms $ \norm\spacedcdot_{B,u}:\left(\calS'\right)_{G_\infty\textup{-smooth}}\to\bbR_{\geq0} $,
	\[ \norm T_{B,u}:=\norm{r'(u)T}_B, \]
	where $ B $ goes over bounded subsets of $ \calS $ and $ u $ goes over $ U(\frakg_\infty) $. 
	But $ \tau $ coincides with the relative topology on  $ \left(\calS'\right)_{G_\infty\textup{-smooth}} $ inherited from $ \calS' $; namely, by \eqref{eq:061} we have that \[ \norm\spacedcdot_{B,u}=\norm\spacedcdot_{r\left(u^\#\right)B}\Big|_{\left(\calS'\right)_{G_\infty\textup{-smooth}}} \] for all $ B $ and $ u $. 
	We note that the set $ r\left(u^\#\right)B $ on the right-hand side is bounded in $ \calS $, or equivalently in $ \calS^L $ for $ L $ such that $ B\subseteq\calS^L $ (see Lemma \hyperref[lem:041:3]{\ref*{lem:041}\ref*{lem:041:3}}), since
	\[ \sup_{f\in r\left(u^\#\right)B}\norm f_{u',-n,c,L}=\sup_{f\in B}\norm {r\left(u^\#\right)f}_{u',-n,c,L}=\sup_{f\in B}\norm f_{u'u^\#,-n,c,L}<\infty \]
	for all $ u'\in U(\frakg_\infty) $, $ n\in\bbZ_{\geq0} $, and $ c\in G(\bbA_f) $. 
	
	\vskip .2in
	
	Thus, $ \left(\calS'\right)_{G_\infty\textup{-smooth}} $ is a complete, hence closed, subspace of $ \calS' $. Since by \cite[Corollary 1 of Lemma 3]{hc} it is also dense in $ \calS' $, it follows that $ \left(\calS'\right)_{G_\infty\textup{-smooth}}=\calS' $.	
\end{proof}

\appendix
\section*{Appendix}\label{appendix}
\renewcommand{\thesection}{A} 
\setcounter{equation}{0}

Here we collect a few facts from functional analysis used in the paper. All vector spaces are assumed to be complex. All locally convex topological vector spaces are assumed to be Hausdorff.

\vskip .2in 

We start by recalling the definition and basic properties of the Gelfand-Pettis integral (e.g., see \cite[\S14]{garrett}).

\vskip .2in

\begin{Thm}\label{thm:148}
	Let $ X $ be a compact Hausdorff space with a Radon measure $ dx $, and let $ E $ and $ F $ be quasi-complete (e.g., complete) locally convex topological vector spaces. Then, for every continuous function $ f:X\to E $ there exists a unique $ \int_Xf(x)\,dx\in E $ (the Gelfand-Pettis integral of $ f $) such that for every continuous linear functional $ T:E\to\bbC $, we have
	\begin{equation}\label{eq:150}
	\scal T{\int_Xf(x)\,dx}=\int_X\scal T{f(x)}\,dx.
	\end{equation}
	Moreover, the following holds:
	\begin{enumerate}[label=\textup{(\arabic*)},leftmargin=*,align=left]
		\item For every continuous linear operator $ A:E\to F $,
		\begin{equation}\label{eq:149}
		A\left(\int_Xf(x)\,dx\right)=\int_XA(f(x))\,dx. 
		\end{equation}
		\item For every continuous seminorm $ \nu:E\to\bbR_{\geq0} $,
		\[ \norm{\int_Xf(x)\,dx}_\nu\leq\int_X\norm{f(x)}_\nu\,dx. \]
	\end{enumerate}
\end{Thm}

\vskip .2in

Next, we recall the definition and basic properties of LF-spaces---strict inductive limits of increasing sequences of Fr\'echet spaces (see e.g. \cite[\S13]{treves}, \cite[\S II.6.3]{schaefer} or \cite[\S12.1]{narici}). 

\vskip .2in 

\begin{Def}\label{def:045}
	Let $ (E_m)_{m\in\bbZ_{>0}} $ be a sequence of Fr\'echet spaces such that $ E_m $ is a closed subspace of $ E_{m+1} $ for every $ m\in\bbZ_{>0} $. The vector space $ E:=\bigcup_{m\in\bbZ_{>0}}E_m $ equipped with the finest locally convex topology with respect to which the inclusion maps $ E_m\hookrightarrow E $ are continuous is called the LF-space with a defining sequence $ (E_m)_{m\in\bbZ_{>0}} $.
\end{Def}

\vskip .2in 

The next lemma uses the notion of bounded sets in a locally convex topological vector space. We recall their definition \cite[Definition 6.1.1 and Theorem 6.1.5]{narici}.

\vskip .2in 

\begin{Def}\label{def:017}
	Let $ E $ be a locally convex topological vector space, and let $ \calC $ be a family of continuous seminorms generating its topology. A subset $ B\subseteq E $ is bounded in $ E $ if the following equivalent conditions hold:
	\begin{enumerate}[label=\textup{(\arabic*)},leftmargin=*,align=left]
		\item For every neighborhood $ U $ of $ 0 $ in $ E $, there exists a $ t_0\in\bbR_{>0} $ such that $ B\subseteq tU $ for all $ t\in\bbC $ such that $ \abs t\geq t_0 $.
		\item $ \sup_{v\in B}\norm v_\rho<\infty $ for all $ \rho\in\calC $.
	\end{enumerate}
\end{Def}

\vskip .2in

We also need the following definition.

\vskip .2in

\begin{Def}\label{def:042}
	Let $ E $ be a vector space.
	The absolutely convex hull of a subset $ A\subseteq E $ is the set
	\[ \aconv(A):=\left\{\sum_{i=1}^na_iv_i:n\in\bbZ_{>0},\ v_i\in A,\ a_i\in\bbC,\ \sum_{i=1}^n\abs{a_i}\leq1\right\}. \]
\end{Def}

\vskip .2in 

\begin{Lem}\label{lem:006}
	Let $ E $ be an LF-space with a defining sequence $ (E_m)_{m\in\bbZ_{>0}} $. Then, we have the following:
	\begin{enumerate}[label=\textup{(\arabic*)},leftmargin=*,align=left]
		\item\label{lem:006:1} The space $ E $ is a complete locally convex topological vector space.
		\item\label{lem:006:2b} For every $ m\in\bbZ_{>0} $, let $ \calU_m $ be a neighborhood basis of $ 0 $ in $ E_m $. Let $ \calU $ be the family of subsets $ U\subseteq E $ of the form
		\[ U=\aconv\left(\bigcup_{m\in\bbZ_{>0}}U_m\right), \]
		where $ U_m\in\calU_m $. Then, $ \calU $ is a neighborhood basis of $ 0 $ in $ E $.
		\item\label{lem:006:3} For every $ m\in\bbZ_{>0} $, $ E_m $ is a closed subspace of $ E $.
		\item\label{lem:006:4} A subset $ B $ of $ E $ is bounded in $ E $ if and only if there exists $ m\in\bbZ_{>0} $ such that $ B\subseteq E_m $ and $ B $ is bounded in $ E_m $.
		\item\label{lem:006:6} Let $ (v_k)_{k\in\bbZ_{>0}}\subseteq E $ and $ v\in E $. Then, $ v_k\to v $ in $ E $ if and only if there exists $ m\in\bbZ_{>0} $ such that $ (v_k)_{k\in\bbZ_{>0}}\subseteq E_m $, $ v\in E_m $, and $ v_k\to v $ in $ E_m $.
		\item\label{lem:006:7} Let $ F $ be a locally convex topological vector space. Then, a linear operator $ A:E\to F $ is continuous if and only if the restrictions $ A\big|_{E_m}:E_m\to F $, $ m\in\bbZ_{>0} $, are continuous.
		\item\label{lem:006:9} A seminorm $ p:E\to\bbR_{\geq0} $ is continuous if and only if the restrictions $ p\big|_{E_m}:E_m\to\bbR_{\geq0} $, $ m\in\bbZ_{>0} $, are continuous.		
	\end{enumerate}
\end{Lem}

\begin{proof}
	The claim \ref{lem:006:1} holds because $ E $ is locally convex by definition (it is Hausdorff by \cite[Theorem 12.1.3(b)]{narici}) and complete by \cite[Theorem 12.1.10]{narici}; 
	\ref{lem:006:2b} follows from \cite[Theorems 12.1.1 and 4.2.11]{narici};
	\ref{lem:006:3} holds by \cite[Theorem 12.1.3(a)]{narici}, \ref{lem:006:4} by \cite[Theorem 12.1.7(a)]{narici}; 
	\ref{lem:006:6} by \cite[Theorem 12.1.7(b)]{narici}, 
	\ref{lem:006:7} by \cite[Theorem 12.2.2]{narici}, 
	and \ref{lem:006:9} follows easily from the definition of topology on $ E $.
\end{proof} 

\vskip .2in 

Next, we recall the notion of the strong dual of a locally convex topological vector space \cite[\S IV.5]{schaefer} (see also \cite[Theorem 3.18]{rudin}).

\vskip .2in 

\begin{Def}\label{def:018}
	Let $ E $ be a locally convex topological vector space. The strong dual $ E' $ of $ E $ is the space of continuous linear functionals $ E\to\bbC $ equipped with the locally convex topology generated by the seminorms $ \norm\spacedcdot_B:E'\to\bbR_{\geq0} $,
	\begin{equation}\label{eq:052}
	 \norm T_B:=\sup_{f\in B}\abs{\scal Tf} 
	\end{equation}
	where $ B $ goes over all bounded sets in $ E $.
\end{Def}

\vskip .2in 

\begin{Lem}[{\cite[Corollary 2 of Theorem 32.2]{treves}}]\label{lem:019}
	Let $ E $ be a metrizable locally convex topological vector space (e.g., a Fr\'echet space) or an LF-space. Then, the strong dual $ E' $ is complete.
\end{Lem}

\vskip .2in 

\begin{Lem}\label{lem:020}
	Let $ E_1,\ldots,E_n $ be locally convex topological vector spaces and for each $ i\in\left\{1,\ldots,n\right\} $, let $ \iota_i $ be the canonical inclusion $ E_i\to\bigoplus_{i=1}^n E_i $. Let us equip the direct sums $ \bigoplus_{i=1}^nE_i $ and $ \bigoplus_{i=1}^nE_i' $ with product topologies. Then, we have the following:
	\begin{enumerate}[label=\textup{(\arabic*)},leftmargin=*,align=left]
		\item\label{lem:020:1} A subset $ B $ of $ \bigoplus_{i=1}^nE_i $ is bounded if and only if $ B\subseteq \bigoplus_{i=1}^n B_i $ for some bounded subsets $ B_i $ of $ E_i $.
		\item\label{lem:020:2} The linear operator
		\[ \Psi:\left(\bigoplus_{i=1}^nE_i\right)'\to\bigoplus_{i=1}^nE_i',\qquad T\mapsto\left(T\circ\iota_i\right)_{i=1}^n, \]
		is an isomorphism of topological vector spaces.
	\end{enumerate}
\end{Lem}

\begin{proof}
	\ref{lem:020:1} This is a special case of \cite[I.5.5]{schaefer}.
	
	\ref{lem:020:2} It is elementary and well-known that $ \Psi $ is a linear isomorphism. It follows easily from \ref{lem:020:1} that it is also a homeomorphism.
\end{proof}

\end{document}